\documentclass[11pt,a4paper]{article}
\usepackage[margin=1in]{geometry}

\usepackage[T1]{fontenc}
\usepackage[utf8]{inputenc}
\usepackage{lmodern}
\usepackage{amsmath,amsthm,amssymb}
\usepackage{graphicx}
\usepackage{tikz}
\usepackage{float}
\usepackage{capt-of}
\usepackage{comment}
\usepackage[numbers,sort&compress]{natbib}
\usepackage[colorlinks=true,citecolor=black,linkcolor=black,urlcolor=blue]{hyperref}
\usepackage[nameinlink,capitalise,noabbrev]{cleveref}

\hypersetup{
  hypertexnames=false,
  pdftitle={Linear Turan Numbers of Uniform Hypertrees},
  pdfauthor={Rajat Adak and Pragya Verma},
  pdfsubject={Extremal combinatorics and linear uniform hypergraphs},
  pdfkeywords={linear Turan number, uniform linear hypergraph, hypertree,
    linear path, cograph, crown, Steiner system}
}
\usetikzlibrary{calc,positioning,quotes}

\theoremstyle{plain}
\newtheorem{theorem}{Theorem}
\newtheorem{lemma}[theorem]{Lemma}
\newtheorem{corollary}[theorem]{Corollary}
\newtheorem{proposition}[theorem]{Proposition}

\newtheorem{claim}[theorem]{Claim}

\theoremstyle{definition}

\newtheorem{conjecture}[theorem]{Conjecture}

\theoremstyle{remark}
\newtheorem{remark}[theorem]{Remark}

\renewcommand{\geq}{\geqslant}
\renewcommand{\leq}{\leqslant}

\renewcommand{\le}{\leqslant}

\crefname{theorem}{Theorem}{Theorems}
\crefname{proposition}{Proposition}{Propositions}
\crefname{lemma}{Lemma}{Lemmas}
\crefname{corollary}{Corollary}{Corollaries}
\crefname{claim}{Claim}{Claims}
\crefname{clm}{Claim}{Claims}
\crefname{cl}{Claim}{Claims}
\crefname{clam}{Claim}{Claims}
\crefname{definition}{Definition}{Definitions}
\crefname{conjecture}{Conjecture}{Conjectures}
\crefname{remark}{Remark}{Remarks}
\crefname{case}{Case}{Cases}

\title{Linear Tur\'an Numbers of Uniform Hypertrees}
\author{%
  Rajat Adak\thanks{Corresponding author. Email: \texttt{rajatadak@iisc.ac.in}.}
  \quad
  Pragya Verma\thanks{Email: \texttt{pragyaverma@iisc.ac.in}.}\\[0.5em]
  {\small Department of Computer Science and Automation,}\\
  {\small Indian Institute of Science, Bengaluru, India.}
}
\date{}

\begin{document}
\maketitle

\begin{abstract}
A hypergraph \(H\) is said to be \emph{linear} if every pair of
vertices is contained in at most one hyperedge. Given a family
\(\mathcal{F}\) of \(r\)-uniform hypergraphs, an \(r\)-uniform
hypergraph \(H\) is said to be \(\mathcal{F}\)-free if it contains no
member of \(\mathcal{F}\) as a subhypergraph. The linear Tur\'an
number \(\operatorname{ex}^{\mathrm{lin}}_r(n,\mathcal{F})\) denotes
the maximum number of hyperedges in an \(n\)-vertex
\(\mathcal{F}\)-free linear \(r\)-uniform hypergraph.

Gy\'arf\'as, Ruszink\'o, and S\'ark\"ozy
initiated the systematic study of linear Tur\'an numbers of acyclic
\(3\)-uniform linear hypergraphs. In this paper, we extend this
direction to higher uniformity, with particular emphasis on linear
hypertrees with four hyperedges.

Our main result concerns the linear path \(P_4^r\). We prove that 
$$
\operatorname{ex}^{\mathrm{lin}}(n,P_4^r)\leq (r+1)n/r,
$$
for every \(r\geq 2\), and characterize equality: the extremal
hypergraphs are precisely the disjoint unions of Steiner systems
\(S(2,r,r^2)\), whenever such systems exist.  We also exhibit
counterexamples to a key structural claim used in a previously proposed
proof of the \(4\)-uniform case.

For the broom \(B_4^r\), obtained from \(S_3^r\) by appending a
hyperedge incident with a degree-one vertex, we prove
$$
\operatorname{ex}^{\mathrm{lin}}_r(n,B_4^r)
\leq (r+1)n/r
$$
and characterize the extremal hypergraphs as disjoint unions of
Steiner systems \(S(2,r,r^2)\), whenever such systems exist.

For the crown \(E_4^r\), we establish a degree-sensitive upper bound.
Let
\(A=\{v\in V(H):d(v)\geq2r\}\),
\(B=\{v\in V(H):d(v)\leq r\}\), and
\(C=\{v\in V(H):r+1\leq d(v)\leq2r-1\}\).
We prove that
$$
\lvert E(H)\rvert\leq
\frac{r}{r-1}\lvert B\rvert+
\frac{2r-1}{r}\lvert C\rvert.
$$
As an immediate consequence, we get
\(\operatorname{ex}^{\mathrm{lin}}_r(n,E_4^r)
\leq(2r-1)n/r\). We also give a lower-bound construction leaving a
constant-factor gap and conjecture that every \(n\)-vertex linear
\(r\)-uniform \(E_4^r\)-free hypergraph satisfies
\(\lvert E(H)\rvert\leq r(n-s)/(r-1)\), where
\(s=\lvert\{v\in V(H):d(v)\geq 2r\}\rvert\). Our degree-sensitive
estimate proves this conjecture whenever no vertex has degree between
\(r+1\) and \(2r-1\).

Finally, we determine the linear Tur\'an number of an \(r\)-uniform
linear star \(S_k^r\), proving that
$$
\operatorname{ex}^{\mathrm{lin}}_r(n,S_k^r)
\leq n(k-1)/r,
$$
with equality precisely for \((k-1)\)-regular linear \(r\)-uniform
hypergraphs, whenever such hypergraphs exist. We also give a
construction showing that, under suitable divisibility and existential
assumptions,
$$
\operatorname{ex}^{\mathrm{lin}}_r(n,T_k^r)
\geq n(k-1)/r
$$
for every linear \(r\)-uniform hypertree \(T_k^r\) with \(k\)
hyperedges.
\end{abstract}



\section{Introduction}
\label{sec:introduction}

Extremal graph theory is largely concerned with Tur\'an-type problems,
which ask for the maximum number of edges in a graph or hypergraph that
avoids a prescribed forbidden substructure. In this work, we focus on
the study of linear Tur\'an numbers for uniform hypergraphs. The
Tur\'an problem for hypergraphs has been investigated extensively; see,
for instance, the surveys
\cite{furedi1991turan,furedi2013history,keevash2011hypergraph}
and the book~\cite{gerbner2018extremal}.

The study of Tur\'an numbers has a long history, beginning with the
classical result of Tur\'an for cliques. Let
\(\operatorname{ex}(n,K_{r+1})\) denote the maximum number of edges in
an \(n\)-vertex graph that does not contain a copy of \(K_{r+1}\) as a
subgraph. The Tur\'an graph, denoted by \(T(n,r)\), is the complete
\(r\)-partite graph on \(n\) vertices whose partite sets are as nearly
equal in cardinality as possible.

\begin{theorem}\label{th:Turan}
\emph{(Tur\'an~\cite{turan1941egy})}
For \(n\geq r\geq 1\),
$$
\operatorname{ex}(n,K_{r+1})
\leq |E(T(n,r))|.
$$
Equality holds if and only if the graph is isomorphic to
\(T(n,r)\).
\end{theorem}

In contrast to graphs, hypergraphs admit several natural notions of
containment and several different analogues of paths, cycles, and
trees. These include Berge, minimal, and linear configurations, among
others. Consequently, even natural extensions of classical
graph-theoretic Tur\'an problems can lead to substantially different
questions in the hypergraph setting. Tur\'an-type results for minimal
paths and cycles were obtained by Mubayi and
Verstra\"ete~\cite{mubayi2007minimal}, while exact results for uniform
linear paths were established by F\"uredi, Jiang, and
Seiver~\cite{furedi2014exact}. For Berge paths, sharp bounds were
obtained by Gy\H{o}ri, Katona, and
Lemons~\cite{gyHori2016hypergraph}, with a remaining case resolved
in~\cite{davoodi2018erdHos}. In this paper, we work exclusively in the
linear framework.

The main result of this paper concerns the linear path \(P_4^r\). We
prove the sharp upper bound
$$
\operatorname{ex}^{\mathrm{lin}}_r(n,P_4^r)
\leq
\frac{(r+1)n}{r}
$$
for every \(r\geq 2\), and characterize all hypergraphs attaining
equality. The proof is based on the line graph of a linear hypergraph.
Four hyperedges form a copy of \(P_4^r\) if and only if the corresponding
vertices induce a path on four vertices in the line graph. Thus, for a
\(P_4^r\)-free linear hypergraph, the line graph is a cograph, and its
structural decomposition can be used to obtain the required bound.

A preliminary version of some of the results in this paper appeared in
the proceedings of the International Workshop on Combinatorial
Algorithms (IWOCA 2026)~\cite{AdakVerma2026}. The present article is
self-contained and substantially expands the scope of that work. The
result for \(P_4^r\) above settles the conjecture from the proceedings
paper \cite{AdakVerma2026} for every uniformity. In addition, we
establish a new degree-sensitive upper bound for the crown \(E_4^r\),
from which the upper bound proved in \cite{AdakVerma2026} follows as an
immediate consequence, and formulate a corresponding degree-sensitive
conjecture. We prove this conjecture when no vertex has degree between
\(r+1\) and \(2r-1\), and give a new affine-plane construction yielding
a lower bound for \(\operatorname{ex}^{\mathrm{lin}}_r(n,E_4^r)\). We
also provide counterexamples to a structural claim used in a previously
proposed proof of the \(4\)-uniform case of \(P_4^r\)
in~\cite{zhang2025linear}.

\subsection{Preliminaries}

A hypergraph \(H=(V(H),E(H))\) consists of a vertex set \(V(H)\) and a
collection \(E(H)\) of subsets of \(V(H)\), called hyperedges. For a
vertex \(v\in V(H)\), the degree \(d_H(v)\), or simply \(d(v)\), is the
number of hyperedges containing \(v\). We denote the minimum and
maximum vertex degrees of \(H\) by \(\delta(H)\) and \(\Delta(H)\),
respectively.

A hypergraph \(H\) is called \(r\)-uniform if every
\(e\in E(H)\) contains exactly \(r\) vertices. It is called
\emph{linear} if any two distinct hyperedges intersect in at most one
vertex.

A linear cycle is a sequence of distinct hyperedges
$e_1,e_2,\ldots,e_t$, where $t\geq 3$,
such that
$e_i\cap e_{i+1}\neq\emptyset$
for every \(i\), where the indices are taken modulo \(t\), and all the
corresponding intersection vertices are distinct. A linear hypergraph
is called \emph{acyclic} if it contains no linear cycle. A connected
acyclic linear \(r\)-uniform hypergraph is called a linear
\(r\)-uniform hypertree.

For a graph \(G\), its \(r\)-uniform \emph{expansion}, denoted by
\(G^r\), is obtained by replacing every graph edge \(uv\in E(G)\) with
an \(r\)-element hyperedge containing \(u\), \(v\), and \(r-2\) new
vertices. The new vertices introduced for distinct graph edges are
chosen to be pairwise disjoint. The resulting hypergraph is
automatically linear and \(r\)-uniform.

\subsection{Linear Tur\'an Numbers}

Let \(\mathcal{F}\) be a family of linear \(r\)-uniform hypergraphs. An
\(r\)-uniform linear hypergraph \(H\) is said to be
\(\mathcal{F}\)-free if it contains no member of \(\mathcal{F}\) as a
subhypergraph. The \emph{linear Tur\'an number}
$$\operatorname{ex}^{\mathrm{lin}}_r(n,\mathcal{F})$$
denotes the maximum number of hyperedges in an \(n\)-vertex
\(\mathcal{F}\)-free linear \(r\)-uniform hypergraph. When
\(\mathcal{F}=\{F\}\), we abbreviate the notation to
\(\operatorname{ex}^{\mathrm{lin}}_r(n,F)\).

The term \emph{linear Tur\'an number} was introduced by
Collier-Cartaino, Graber, and
Jiang~\cite{COLLIER-CARTAINO_GRABER_JIANG_2018}, who established
foundational bounds for linear cycles. However, related extremal
questions appeared much earlier. The celebrated result of Ruzsa and
Szemer\'edi~\cite{ruzsa1978triple} can be expressed in the form
$$
n^{\,2-\frac{c}{\sqrt{\log n}}}
\leq
\operatorname{ex}^{\mathrm{lin}}_3(n,C_3^3)
=
o(n^2),
$$
where \(c>0\) is a constant and \(C_3^3\) denotes the \(3\)-uniform
expansion of the triangle \(C_3\).

The systematic study of acyclic structures in this setting was
initiated by Gy\'arf\'as, Ruszink\'o, and
S\'ark\"ozy~\cite{gyarfas2022linear}, who investigated acyclic
\(3\)-uniform linear hypergraphs and obtained exact values and bounds
for several forbidden configurations. Zhang and
Wang~\cite{zhang2025linear} subsequently considered corresponding
questions in the \(4\)-uniform setting. In particular, the sharp bound
for the \(3\)-uniform linear path \(P_4^3\) was obtained
in~\cite{gyarfas2022linear}, while a proof for the \(4\)-uniform case
was proposed in~\cite{zhang2025linear}. Our result for \(P_4^r\)
removes the restriction on the uniformity.

The subject has continued to develop rapidly. Zhou and
Yuan~\cite{zhou2025turan} obtained general \(r\)-uniform bounds,
including
$$
\operatorname{ex}^{\mathrm{lin}}_r(n,P_k^r)
\leq
\frac{(2r-3)kn}{2}
$$
for \(r\geq 3\) and \(k\geq 4\). In another direction, Khormali and
Palmer~\cite{khormali2022turan}, and later Zhang
et al.~\cite{zhang2025turan}, investigated Tur\'an numbers of
hypergraph star forests.

Even linear hypertrees with only a few hyperedges exhibit several
different extremal behaviours. For two hyperedges, the only connected
linear hypertree is \(P_2^r\), where \(P_k\) denotes the path with
\(k\) edges. It is immediate that
$$
\operatorname{ex}^{\mathrm{lin}}_r(n,P_2^r)
=
\left\lfloor \frac{n}{r}\right\rfloor.
$$
For three hyperedges, the two possibilities are \(P_3^r\) and
\(S_3^r\), where \(S_k\) denotes the star with \(k\) edges. It was
shown in~\cite{zhang2025linear} that
$$
\operatorname{ex}^{\mathrm{lin}}_r(n,P_3^r)\leq n.
$$

For four hyperedges, the expansions of the three trees with four edges
are \(P_4^r\), \(S_4^r\), and \(B_4^r\), where \(B_4\) is obtained from
\(S_3\) by attaching an additional edge to one of its leaves. Alongside
these expansion configurations, we consider the crown \(E_4^r\), which
is a natural non-expansion linear hypertree. The four configurations
studied in this paper are illustrated in
Figure~\ref{fig:four-edge-configurations}.

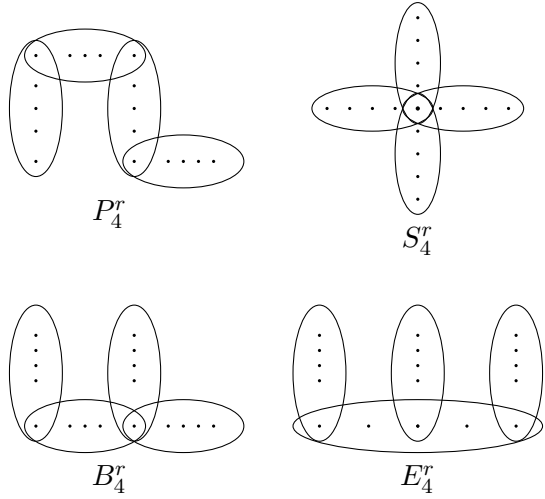
\begin{figure}
\centering
\begin{tikzpicture}[scale=0.5]

\begin{scope}[xshift=0cm, yshift=7cm]

    \draw (0,0) ellipse (0.7 and 1.8);
    \draw (1.3,1.4) ellipse (1.6 and 0.7);
    \draw (2.6,0) ellipse (0.7 and 1.8);
    \draw (3.9,-1.4) ellipse (1.6 and 0.7);

    \foreach \y in {1.4,0.6,0.0,-0.6,-1.4}
      \fill (0,\y) circle (1.2pt);

    \foreach \y in {1.4,0.6,0.0,-0.6,-1.4}
      \fill (2.6,\y) circle (1.2pt);

    \foreach \x in {0.9,1.3,1.7}
      \fill (\x,1.4) circle (1.2pt);

    \foreach \x in {3.5,3.9,4.3,4.7}
      \fill (\x,-1.4) circle (1.2pt);

    \node at (1.95,-2.8) {$P_4^r$};
\end{scope}

\begin{scope}[xshift=7.5cm, yshift=7cm]

    \draw (2.6,1.2) ellipse (0.6 and 1.6);
    \draw (2.6,-1.2) ellipse (0.6 and 1.6);
    \draw (1.4,0) ellipse (1.6 and 0.6);
    \draw (3.8,0) ellipse (1.6 and 0.6);

    \fill (2.6,0) circle (1.6pt);

    \foreach \y in {0.6,1.2,1.8,2.4}
      \fill (2.6,\y) circle (1.2pt);

    \foreach \y in {-0.6,-1.2,-1.8,-2.4}
      \fill (2.6,\y) circle (1.2pt);

    \foreach \x in {2.0,1.4,0.8,0.2}
      \fill (\x,0) circle (1.2pt);

    \foreach \x in {3.2,3.8,4.4,5.0}
      \fill (\x,0) circle (1.2pt);

    \node at (2.6,-3.5) {$S_4^r$};
\end{scope}

\begin{scope}[xshift=0cm, yshift=0cm]

    \draw (0,0) ellipse (0.7 and 1.8);
    \draw (2.6,0) ellipse (0.7 and 1.8);
    \draw (1.3,-1.4) ellipse (1.6 and 0.7);
    \draw (3.9,-1.4) ellipse (1.6 and 0.7);

    \foreach \y in {1,0.6,0.2,-0.2,-1.4}
      \fill (0,\y) circle (1.2pt);

    \foreach \y in {1,0.6,0.2,-0.2,-1.4}
      \fill (2.6,\y) circle (1.2pt);

    \foreach \x in {0.9,1.3,1.7}
      \fill (\x,-1.4) circle (1.2pt);

    \foreach \x in {3.5,3.9,4.3,4.7}
      \fill (\x,-1.4) circle (1.2pt);

    \node at (1.95,-2.8) {$B_4^r$};
\end{scope}

\begin{scope}[xshift=7.5cm, yshift=0cm]

    \draw (0,0) ellipse (0.7 and 1.8);
    \draw (2.6,0) ellipse (0.7 and 1.8);
    \draw (5.2,0) ellipse (0.7 and 1.8);
    \draw (2.6,-1.4) ellipse (3.3 and 0.7);

    \foreach \x in {0,2.6,5.2} {
      \foreach \y in {1,0.6,0.2,-0.2,-1.4}
        \fill (\x,\y) circle (1.2pt);
    }

    \foreach \x in {1.3,2.6,3.9}
      \fill (\x,-1.4) circle (1.2pt);

    \node at (2.6,-2.8) {$E_4^r$};
\end{scope}

\end{tikzpicture}
\caption{The four linear \(r\)-uniform configurations with four
hyperedges studied in this paper. The diagrams are schematic: the
unlabelled points in each hyperedge represent its remaining vertices.}
\label{fig:four-edge-configurations}
\end{figure}

The crown \(E_4^r\) consists of three pairwise disjoint hyperedges
together with a fourth hyperedge, called the \emph{base}, that
intersects each of the three disjoint hyperedges. Unlike
\(P_4^r\), \(S_4^r\), and \(B_4^r\), the crown need not arise as the
expansion of a graph tree.

Crowns have been studied in several recent
works~\cite{carbonero2021crowns,carbonero2022crowns,tang2022turan,zhang2025generalized}. Fletcher~\cite{fletcher2021improved} proved
$$
\operatorname{ex}^{\mathrm{lin}}_3(n,E_4^3)
<
\frac{5n}{3}
$$
for the \(3\)-uniform crown. Tang, Wu, Zhang, and
Zheng~\cite{tang2022turan} subsequently obtained an asymptotically
sharp result for the \(3\)-uniform case. Generalized crowns in
higher-uniformity linear hypergraphs were studied by Zhang, Broersma,
and Wang~\cite{zhang2025generalized} and Adak~\cite{adak2026upper}.

More generally, we use \(T_k^r\) to denote a linear \(r\)-uniform
hypertree with \(k\) hyperedges. Such a hypertree need not arise as the
expansion of a graph tree. Determining the linear Tur\'an number of an
arbitrary hypertree appears difficult even when the number of
hyperedges is small. This reflects the difficulty of the corresponding
problem for ordinary graphs, where the Erd\H{o}s--S\'os
conjecture~\cite{erdos1965extremal} concerns the Tur\'an number of an
arbitrary tree. We therefore focus on general lower-bound constructions
and on the linear hypertrees with four hyperedges described above.

\subsection{Steiner Systems}

A Steiner system \(S(t,r,n)\) consists of an \(n\)-element vertex set
together with a collection of \(r\)-element blocks such that every
\(t\)-element subset of the vertex set is contained in exactly one
block. In particular, an \(S(2,r,n)\) is a linear \(r\)-uniform
hypergraph in which every pair of vertices is contained in exactly one
hyperedge.

Steiner systems \(S(2,r,n)\) arise naturally in linear Tur\'an
problems because they attain the pair-counting upper bound for the
number of hyperedges in a linear \(r\)-uniform hypergraph. Steiner
triple systems \(S(2,3,n)\), also denoted by \(STS(n)\), have appeared
repeatedly in the study of linear Tur\'an numbers for triple systems;
see
\cite{gyarfas2022linear,gyarfas2021turan,gyarfas2022linearwicket,sarkozy2023turan}.
Steiner \(2\)-designs with block size \(4\), namely \(S(2,4,n)\),
play an analogous role in the \(4\)-uniform setting, for instance
in~\cite{zhang2025linear}. These systems provide extremal or
near-extremal constructions for several linear Tur\'an-type problems.

We record the standard edge and degree counts for later use.

\begin{lemma}\label[lemma]{steinercount}
Let \(H\) be a Steiner system \(S(2,r,n)\), assuming that such a system
exists. Then
$$
|E(H)|
=
\frac{\binom{n}{2}}{\binom{r}{2}}
=
\frac{n(n-1)}{r(r-1)},
$$
and
$$
d(v)=\frac{n-1}{r-1}
\qquad
\text{for every }v\in V(H).
$$
\end{lemma}

\begin{proof}
Every hyperedge contains \(\binom{r}{2}\) unordered pairs of vertices,
and every pair among the \(\binom{n}{2}\) vertex-pairs is contained in
exactly one hyperedge. Consequently,
$$
|E(H)|\binom{r}{2}
=
\binom{n}{2},
$$
which gives
$$
|E(H)|
=
\frac{\binom{n}{2}}{\binom{r}{2}}
=
\frac{n(n-1)}{r(r-1)}.
$$

Now fix \(v\in V(H)\). Each of the \(n-1\) pairs
\(\{v,u\}\), where \(u\neq v\), is contained in exactly one
hyperedge. Every hyperedge containing \(v\) accounts for exactly
\(r-1\) such pairs. Therefore,
$$
d(v)(r-1)=n-1,
$$
and hence
$$
d(v)=\frac{n-1}{r-1}.
$$
\end{proof}

The formulas in Lemma~\ref{steinercount} imply the necessary
divisibility conditions
$$
r-1\mid n-1
\qquad\text{and}\qquad
\binom{r}{2}\mid\binom{n}{2}
$$
for the existence of an \(S(2,r,n)\). For every fixed \(r\), these
necessary divisibility conditions are also sufficient for all
sufficiently large admissible values of \(n\), by the existence theorem
for designs~\cite{keevash2014existence}.

For small block sizes, more precise existence results are known. A
Steiner triple system \(S(2,3,n)\) exists if and only if
$$
n\equiv 1,3\pmod 6,
$$
while a Steiner system \(S(2,4,n)\) exists if and only if
$$
n\equiv 1,4\pmod{12}.
$$
Moreover, when \(r\) is a prime power, an affine plane of order \(r\)
gives a Steiner system \(S(2,r,r^2)\). In general, however, the
existence of \(S(2,r,r^2)\) is not guaranteed. The
Bruck--Ryser--Chowla theorem
\cite{bruck1949nonexistence,chowla1950combinatorial} gives necessary
conditions for the existence of finite projective planes and rules out
certain orders; for example, no affine plane, and hence no
\(S(2,6,36)\), exists.

A complete characterization of the parameters \((r,n)\) for which a
Steiner system \(S(2,r,n)\) exists is not known. The statements of our
main results are collected in the next section.
\section{Our Results}
\label{sec:our-results}

We begin by determining the exact extremal value for linear stars.
\begin{proposition}\label[proposition]{starprop}
Let \(r\geq 3\) and \(k\geq 2\). Then
$$
\operatorname{ex}^{\mathrm{lin}}_r(n,S_k^r)
\leq
\frac{n(k-1)}{r}.
$$
Equality holds if and only if the extremal linear \(r\)-uniform
hypergraph is \((k-1)\)-regular, whenever such a hypergraph exists.
\end{proposition}

To complement this result, we give a matching lower-bound construction
for arbitrary linear hypertrees under suitable divisibility and
design-existence assumptions.

\begin{theorem}\label{lowerbound}
Let \(r\geq 3\) and \(k\geq 2\), and set
$$
t=(r-1)(k-1)+1.
$$
Suppose that \(t\mid n\) and that a Steiner system \(S(2,r,t)\) exists.
Then
$$
\operatorname{ex}^{\mathrm{lin}}_r(n,T_k^r)
\geq
\frac{n(k-1)}{r}.
$$
This bound is sharp when \(T_k^r\cong S_k^r\).
\end{theorem}

We next consider linear hypertrees with four hyperedges. For the
configuration \(B_4^r\), we determine the exact extremal value and
characterize the extremal hypergraphs.

\begin{theorem}\label{b4thm}
Let \(r\geq 3\). Then
$$
\operatorname{ex}^{\mathrm{lin}}_r(n,B_4^r)
\leq
\frac{(r+1)n}{r}.
$$
Equality holds if and only if the extremal hypergraph is a disjoint
union of Steiner systems \(S(2,r,r^2)\), whenever such systems exist.
\end{theorem}

We then turn to the crown \(E_4^r\). Our principal upper bound is
degree-sensitive and distinguishes vertices of low, intermediate, and
high degree.

\begin{theorem}\label{crowndeg}
Let \(H\) be a linear \(r\)-uniform \(E_4^r\)-free hypergraph, where
\(r\geq 3\). Define
$$
\begin{aligned}
A&=\{v\in V(H):d(v)\geq 2r\},\\
B&=\{v\in V(H):d(v)\leq r\},\\
C&=\{v\in V(H):r+1\leq d(v)\leq 2r-1\}.
\end{aligned}
$$
Then
$$
|E(H)|
\leq
\frac{r}{r-1}|B|
+
\frac{2r-1}{r}|C|.
$$
Equivalently, writing \(s=|A|\), we have
$$
|E(H)|
\leq
\frac{r(n-s)}{r-1}
+
\frac{r^2-3r+1}{r(r-1)}|C|.
$$
\end{theorem}

The degree-sensitive estimate immediately implies the upper bound
obtained in the preliminary conference version~\cite{AdakVerma2026}.
Consequently, its proof also provides a shorter proof of that result.

\begin{corollary}\label[corollary]{E4thm}
For every \(r\geq 3\),
$$
\operatorname{ex}^{\mathrm{lin}}_r(n,E_4^r)
\leq
\frac{(2r-1)n}{r}.
$$
\end{corollary}

We complement these upper bounds with a construction for the crown.
Zhang, Broersma, and Wang~\cite{zhang2025generalized} introduced a
construction giving a lower bound for a generalized crown. We show
that, under the appropriate existence assumption, the same
construction is \(E_4^r\)-free.

\begin{theorem}\label{E4lower}
Let \(r\geq 3\), and suppose that an affine plane of order \(r-1\)
exists. If
$(r-1)^2\mid(n-r),
$ then
$$
\operatorname{ex}^{\mathrm{lin}}_r(n,E_4^r)
\geq
\frac{r(n-r)}{r-1}.
$$
\end{theorem}

For every fixed \(r\), the upper bound in \Cref{E4thm} and the lower
bound in \Cref{E4lower} are both linear in \(n\), but their leading
constants differ. The lower-bound construction, together with the known
result for \(r=3\), motivates the following degree-sensitive
conjecture.

\begin{conjecture}\label[conjecture]{crownconj}
Let \(H\) be an \(n\)-vertex linear \(r\)-uniform \(E_4^r\)-free
hypergraph, where \(r\geq 3\), and let
$$
s=\bigl|\{v\in V(H):d(v)\geq 2r\}\bigr|.
$$
Then
$$
|E(H)|\leq \frac{r(n-s)}{r-1}.
$$
\end{conjecture}

For \(r=3\), \Cref{crownconj} follows from the result of Tang, Wu,
Zhang, and Zheng~\cite{tang2022turan}. The bound in \Cref{crowndeg}
shows that vertices whose degrees lie between \(r+1\) and \(2r-1\) are
the only obstruction, within the present weighting argument, to the
conjectured estimate. In particular, we obtain the following special
case.

\begin{corollary}\label[corollary]{crowncor}
Let \(H\) be a linear \(r\)-uniform \(E_4^r\)-free hypergraph such that
\(d(v)\notin\{r+1,r+2,\ldots,2r-1\}\) for every \(v\in V(H)\). If
$$
s=\bigl|\{v\in V(H):d(v)\geq 2r\}\bigr|,
$$
then
$$
|E(H)|\leq \frac{r(n-s)}{r-1}.
$$
\end{corollary}

The construction in \Cref{E4lower} supports \Cref{crownconj}. Indeed,
in a sufficiently large copy of the construction, the \(r\) added
vertices are precisely the vertices of degree at least \(2r\), while
every remaining vertex has degree \(r\). The construction then attains
equality in the conjectured estimate. Moreover, disjoint unions of such
constructions may have \(s>r\) and continue to attain equality.

We finally consider the linear path \(P_4^r\). The known results for
\(r=2\) from the Erd\H{o}s--Gallai theorem~\cite{gallai1959maximal}
and for \(r=3\) from~\cite{gyarfas2022linear}, together with the
lower-bound construction in the preliminary version of this
paper~\cite{AdakVerma2026}, suggested that \((r+1)n/r\) is the sharp
general upper bound. We settle this problem for every uniformity and
characterize all cases of equality.

\begin{theorem}\label{P4thm}
Let \(r\geq 2\). Then
$$
\operatorname{ex}^{\mathrm{lin}}_r(n,P_4^r)
\leq
\frac{(r+1)n}{r}.
$$
Equality holds if and only if the extremal hypergraph is a disjoint
union of Steiner systems \(S(2,r,r^2)\), whenever such systems exist.
\end{theorem}

Zhang and Wang~\cite{zhang2025linear} previously claimed the
\(4\)-uniform case. Their proof uses a structural assertion about a
maximum star that is not valid in general. We provide counterexamples
to that assertion after proving \Cref{P4thm}.
    \section{Proof of \Cref{starprop} and \Cref{lowerbound}}
    The results in this section appeared in the preliminary conference
version~\cite{AdakVerma2026}. We include their proofs to keep the
present article self-contained.

 \vspace{2mm}
 \noindent\textit{Proof of \Cref{starprop}.}
        If $H$ is $S_k^r$-free, then $d(v) \leq k-1$ for all $v \in V(H)$. Since each edge consists of $r$ vertices, we get an analogue of the handshaking lemma as $\sum_{v \in V(H)}d(v) = r|E(H)|$. Thus,
        $r|E(H)| \leq n(k-1) \implies |E(H)| \leq \dfrac{n(k-1)}{r}$.
        Therefore we get $ex_r^{\mathrm{lin}}(n , S_k^r) \leq \dfrac{n(k-1)}{r}$ and clearly for equality we need $d(v) = k-1$ for all $v \in V(H)$. Note that a linear $r$-uniform $(k-1)$-regular hypergraph may not exist. In such cases, equality is not feasible. Clearly, $r \mid n(k-1)$ is a necessary condition. Also it is necessary to have,
        $$|E(H)|\binom{r}{2} \leq \binom{n}{2} \implies (k-1)(r-1) \leq (n-1)$$
        These conditions are not sufficient. For example, when $(n, k, r) = (10, 4, 3)$, 
both the above conditions are satisfied, yet no such hypergraph 
exists. In fact, there is no known general characterization for the existence 
of regular linear $r$-uniform hypergraphs.

    Now we give a constructive proof for \Cref{lowerbound} under the divisibility and existence assumptions.
    \vspace{2mm}
    
    \noindent\textit{Proof of \Cref{lowerbound}.}
Let $n = qt$ and partition the $n$-vertex set into $q$ disjoint parts
$V(H) = V_1 \sqcup V_2 \sqcup \cdots \sqcup V_q$,
such that $|V_i|=t$.
Suppose each $V_i$ induces $H_i$, a component of $H$, and each $H_i$ is a copy of $S(2,r,t)$.

Since each $H_i$ is linear and $r$-uniform by definition, and edges of disjoint components are disjoint, we have $H$ to be linear and $r$-uniform.

By \Cref{steinercount}, we have
\begin{align*}
|E(H)| &= \sum_{i=1}^q |E(H_i)|
= \frac{n}{t}\cdot \frac{t(t-1)}{r(r-1)} \\
&= \frac{n(t-1)}{r(r-1)}
= \frac{n(r-1)(k-1)}{r(r-1)} = \frac{n(k-1)}{r}
\end{align*}
\begin{claim}
    Every linear $r$-uniform hypertree with $k$ edges has exactly $(r-1)k+1$ vertices.
\end{claim}
\begin{proof}
Starting with one edge, clearly it contains $r$ vertices.
When we add a new edge in a linear $r$-uniform hypertree, it intersects the existing vertex set in exactly one
vertex, hence contributes precisely $r-1$ new vertices.
After adding $k-1$ further edges, the total vertex count is
$r + (k-1)(r-1) = (r-1)k + 1$.
\end{proof}
Suppose there exists a $T^r_k$ in $H$, clearly it must be in some $H_i$. From the above claim $|V(H_i)| \geq (r-1)k +1 = \bigl((r-1)(k-1)+1\bigr) + (r-1) = t + (r-1)$. Thus we get a contradiction.

Therefore, $H$ is a $T^r_k$-free graph and $|E(H)|  = \frac{n(k-1)}{r}$. Thus, under the given divisibility and existence assumptions we have $ex_r^{\mathrm{lin}}(n,T_k^r) \geq \frac{n(k-1)}{r}$. From \cref{starprop} it is clear that the bound is sharp when $T^r_k \cong S^r_k$.  
    \section{Proof of \Cref{b4thm}}
The result in \Cref{b4thm} appeared in the preliminary conference
version~\cite{AdakVerma2026}. Its proof is included here for
self-containment.\subsection{Proof of inequality}\label{ineqb4}
        Suppose the statement in \Cref{b4thm} is false. Assume $H$ to be a minimal counterexample. That is, $H$ is a $B_4^r$-free linear $r$-uniform  hypergraph on $n$ vertices and   
        $$|E(H)| > \frac{(r+1)n}{r}$$
        and for any subhypergraph $H'$ of $H$, \Cref{b4thm} holds for $H'$.
        \begin{claim}
            $\delta(H) \geq 2$.        \end{claim}
            \begin{proof}
                If $H$ contains an isolated vertex, then we can drop that vertex, resulting in a smaller counterexample. Thus, suppose there exists $v \in V(H)$, such that $d(v) =1$. Let $H' = H \setminus \{v\}$. Clearly, $$|E(H')| = |E(H)| -1 > \frac{(r+1)n}{r} -1 > \frac{(r+1)n}{r} - \frac{(r+1)}{r} = \frac{(r+1)(n-1)}{r}$$
                Thus $H'$ is a counterexample for \Cref{b4thm}, contradicting the minimality of $H$.
            \end{proof}

                    $H$ must be connected, otherwise, we will get at least one connected component of $H$ as a counterexample, contradicting minimality of $H$.
        \begin{claim}
            $\Delta(H) \geq (r+2)$.
            \begin{proof}
                Suppose $\Delta(H) \leq (r+1)$, then $d(v) \leq (r+1)$ for all $v \in V(H)$. Therefore, $\sum _{v \in V(H)}d(v) \leq n(r+1)$. Since $H$ is $r$-uniform, we have 
                $$\sum_{v\in V(H)}d(v) = r|E(H)| > (r+1)n$$
                Thus we get a contradiction.
            \end{proof}
        \end{claim}
        Let $v \in V(H)$ be such that $d(v) = \Delta(H) = k$. Thus we get an expanded star $S=S_k^r$ centered at $v$. Let the edges of $S$ be $E(S)=\{e_1,e_2.\dots,e_k\}$. Let $u \in e_1\setminus \{v\}$. Since $\delta(H) \geq 2$, there exists $f \in E(H)$ containing $u$ such that $f \neq e_1$. Suppose $f$ intersects $t$ many hyperedges from $E(S)$, without loss of generality let these edges be $\{e_1,e_2,\dots,e_t\}$. Since $H$ is $r$-uniform, $t \leq r$.
        
        Note that, $k \geq r+2$, therefore $t \leq k-2$ and thus $e_{t+1},e_{t+2} \in E(S)$. It is easy to see that the edges $f,e_1,e_{t+1},e_{t+2}$ form a $B_4^r$. Thus we get a contradiction. Therefore, no such counterexample exists.
    \subsection{Characterizing extremal Hypergraphs}
        For characterizing the extremal hypergraphs, we will only focus on the cases when $S(2,r,r^2)$ exists. If $S(2,r,r^2)$ does not exist, then the bound in \Cref{b4thm} is not tight.
    \newline\textbf{$\bullet$ If Part}
    
\noindent Assume $H$ to be a disjoint union of copies of a Steiner system $S(2,r,r^2)$.
Fix one component $C\cong S(2,r,r^2)$ with $|V(C)|=r^2$.
\begin{claim}
    $|E(C)| = r(r+1)$.
\end{claim}
\begin{proof}
    From \Cref{steinercount}, we have $$|E(C)| = \frac{\binom{r^2}{2}}{\binom{r}{2}} = \frac{r^2(r^2-1)}{r(r-1)} = r(r+1)$$
\end{proof}
If $H$ has $t$ components, then $n=t r^2$ and $|E(H)| = t\cdot r(r+1)=\frac{(r+1)n}{r}$.

\begin{lemma}\label[lemma]{lem:steiner}
    If $C \cong S(2,r,r^2)$, and $e \in E(C)$, then for any $v\in V(C)$ such that $v \notin e$, there exists a unique $f \in E(C)$ containing $v$ such that $f\cap e = \emptyset$.
\end{lemma}
\begin{proof}
Fix an edge $e\in E(C)$ and a vertex $v\in V(C)\setminus e$. From \Cref{steinercount} we have $d(v) = r+1$. Hence there are exactly $r+1$ edges of $C$ containing $v$.

For each vertex $u\in e$, the pair $\{u,v\}$ is contained in a unique edge, which we denote by $e_u$. Clearly, if $u\neq u'$, then $e_u\neq e_{u'}$. Hence, the $r$ vertices of $e$ give rise to $r$ distinct edges containing $v$, each intersecting $e$ in exactly one vertex.

Since $v$ lies in exactly $r+1$ edges in total, there exists precisely one further edge $f$ containing $v$ that is different from all the $e_u$. This edge $f$ cannot intersect $e$; otherwise it would coincide with $e_u$ for some $u\in e$. Therefore, $f\cap e=\emptyset$.
\end{proof}

\begin{claim}
    $C$ is $B_4^r$-free.
\end{claim}
\begin{proof}
    Suppose there exists a $B_4^r$ in $C$ and assume it consists of edges $e_1,e_2,e_3$ and $e_4$, where $e_1\cap e_2\cap e_3 = \{v\}$ and $e_3\cap e_4 \neq \emptyset$. Clearly $v \notin e_4$ and $e_1,e_2$ contain $v$ and are disjoint from $e_4$. Thus we get a contradiction to \Cref{lem:steiner}.
\end{proof}
Thus, $|E(H)| = \dfrac{(r+1)n}{r}$ and $H$ is $B_4^r$-free.

\noindent\textbf{$\bullet$ Only If Part}

\noindent Now let $H$ be a $B_4^r$-free linear $r$-uniform hypergraph on $n$ vertices satisfying $|E(H)|=\dfrac{(r+1)n}{r}$. Without loss of generality, assume $H$ is connected. We need to show that $H$ is $S(2,r,r^2)$.
\begin{claim}
    $H$ is $(r+1)$-regular
\end{claim}
\begin{proof}
    Since $H$ is $r$-uniform, $\sum_{v\in V(H)} d(v) = r|E(H)| = (r+1)n$. Suppose there exists a vertex $v \in V(H)$ such that $d(v) = k \geq r+2$. Now consider the $S_k^r$ centered at $v$. It is easy to see that there exists a vertex in $S_k^r$, other than $v$, with degree at least $2$, otherwise $H \cong S_k^r$, which will result, $n = k(r-1) +1$ and $|E(H)| = k = \frac{n-1}{r-1} < \frac{(r+1)n}{r}$, a contradiction.
    
    Thus, using the arguments as in the proof of inequality in \cref{ineqb4}, we will get a $B_4^r$ in $H$.

    Thus every vertex has degree at most $r+1$ and the average is $r+1$, forcing
$d(v)=r+1$ for all $v\in V(H)$, i.e.\ $H$ is $(r+1)$-regular.
\end{proof}

Let $v \in V(H)$ and the $r+1$ edges containing $v$ be $e_1,\dots,e_{r+1}$.
By linearity, these edges are pairwise disjoint outside $v$, so the set
$$
X:=\{v\}\ \cup\ \bigcup_{i=1}^{r+1}(e_i\setminus\{v\})
$$
has size, $|X| = 1 + (r+1)(r-1) = r^2$.

\begin{claim}
$V(H)=X$.
\end{claim}

\begin{proof}
Suppose for contradiction $ V(H)\setminus X \neq \emptyset$.
Since $H$ is connected, there is an edge $f\in E(H)$ meeting $X$ and containing at least one
vertex outside $X$. Choose such an $f$.
Clearly, $v\notin f$.
Let $u\in f\cap X$ (necessarily $u\neq v$). Without loss of generality assume $u \in e_1$.

Now $f$ has size $r$, and contains at least one vertex outside $X$.
So among the remaining vertices of $f$, at most $r-1$ lie in $X$.
Each vertex of $f \cap \{X\setminus\{v\}\}$ lies in \emph{exactly one} of the star edges
$e_1,e_2,\dots,e_{r+1}$, so $f$ can intersect at most $r-1$ of these $r+1$ edges
(including $e_1$). Therefore there exist \emph{two} distinct edges,
say $e_a$ and $e_b$ among $\{e_2,\dots,e_{r+1}\}$, that are disjoint from $f$.

But then the hyperedges, $f,e_1,e_a,e_b$ form a copy of $B_4^r$, contradicting that $H$ is $B_4^r$-free.
\end{proof}

Finally, we show $H$ is a Steiner system $S(2,r,r^2)$.
Because $H$ is linear, any unordered pair of vertices is contained in \emph{at most one} edge.
Therefore, counting vertex-pairs inside edges gives
$$
|E(H)|\binom{r}{2} \le \binom{r^2}{2}.
$$
But substituting $|E(H)|=\dfrac{(r+1)n}{r} = \dfrac{(r+1)r^2}{r}= r(r+1)$ yields equality:
$$
r(r+1)\binom{r}{2}
= r(r+1)\cdot \frac{r(r-1)}{2}
= \frac{r^2(r^2-1)}{2}
= \binom{r^2}{2}.
$$
Hence every pair of vertices in $H$ lies in \emph{exactly one} edge, i.e.\ $H$ is $S(2,r,r^2)$.

\section{Proofs of the results for the crown \(E_4^r\)}

In this section, we prove the degree-sensitive upper bound for the
crown \(E_4^r\). The upper bound obtained in the preliminary conference
version then follows immediately as a corollary. We subsequently prove
the lower-bound construction in \Cref{E4lower}, and then prove the
special case of \Cref{crownconj} stated in \Cref{crowncor}.

\subsection{The degree-sensitive upper bound, \Cref{crowndeg}}

We begin with a crown-forcing observation.

\begin{lemma}
\label[lemma]{crownlem}
Let \(H\) be a linear \(r\)-uniform hypergraph, where \(r\geq3\), and
let \(e\in E(H)\). Suppose that \(e\) contains three distinct vertices
\(x,y,z\) satisfying,
$d(x)\geq2r,
d(y)\geq r+1,
d(z)\geq2.
$
Then \(e\) is the base edge of a copy of \(E_4^r\).
\end{lemma}

\begin{proof}
Since \(d(z)\geq2\), there exists an edge
$$
f\in E(H)\setminus\{e\}
$$
such that \(z\in f\). We first find an edge through \(y\) that is disjoint from \(f\). There
are at least
$$
d(y)-1\geq r
$$
edges in \(E(H)\setminus\{e\}\) containing \(y\).

No such edge can contain \(z\), since it would then intersect \(e\) in
both \(y\) and \(z\), contradicting the linearity of \(H\). Moreover,
for every vertex \(u\in f\setminus\{z\}\), at most one edge containing
\(y\), other than \(e\), can also contain \(u\). Indeed, two such edges
would intersect in both \(y\) and \(u\), again contradicting
linearity.

Consequently, at most
$|f\setminus\{z\}|=r-1
$ edges containing \(y\), other than \(e\), can intersect \(f\). Since
there are at least \(r\) such edges, there exists an edge
$
g\in E(H)\setminus\{e\}
$
such that
$$
y\in g
\qquad\text{and}\qquad
f\cap g=\emptyset.
$$

We now find an edge through \(x\) that is disjoint from both \(f\) and
\(g\). There are at least
$$
d(x)-1\geq2r-1
$$
edges containing \(x\), other than \(e\).

By the same linearity argument, at most \(r-1\) of these edges can
intersect \(f\), and at most \(r-1\) can intersect \(g\). Thus, at most
$
2(r-1)=2r-2
$ edges containing \(x\), other than \(e\), intersect \(f\cup g\).
Since $
2r-1>2r-2,$
there exists an edge
$
h\in E(H)\setminus\{e\}
$
such that
$$
x\in h,\qquad
h\cap f=\emptyset,
\qquad\text{and}\qquad
h\cap g=\emptyset.
$$

Therefore, \(f,g,h\) are pairwise disjoint and satisfy
$$
z\in e\cap f,\qquad
y\in e\cap g,\qquad
x\in e\cap h.
$$
Hence \(e,f,g,h\) form a copy of \(E_4^r\) with base edge \(e\).
\end{proof}

We now prove the principal crown upper bound.

\begin{proof}[Proof of \Cref{crowndeg}]
Recall the partition
$$
\begin{aligned}
A&=\{v\in V(H):d(v)\geq2r\},\\
B&=\{v\in V(H):d(v)\leq r\},\\
C&=\{v\in V(H):r+1\leq d(v)\leq2r-1\}.
\end{aligned}
$$
Thus,
$
V(H)=A\mathbin{\dot\cup}B\mathbin{\dot\cup}C.
$

We may assume that \(H\) has no isolated vertices. Indeed, deleting
isolated vertices does not change \(E(H)\), \(A\), or \(C\), and can
only decrease \(|B|\). Therefore, proving the desired inequality after
deleting all isolated vertices also proves it for the original
hypergraph.

For every vertex \(v\in V(H)\), define
$$
\omega(v)=
\begin{cases}
\displaystyle
\frac{r}{(r-1)d(v)},&v\in B,\\[2ex]
\displaystyle
\frac{2r-1}{r\,d(v)},&v\in C,\\[2ex]
0,&v\in A.
\end{cases}
$$
For each edge \(e\in E(H)\), define
$$
\omega(e)=\sum_{v\in e}\omega(v).
$$

We claim that
$$
\omega(e)\geq1
\qquad
\text{for every }e\in E(H).
\label{eq:crown-edge-weight}
$$

Fix an arbitrary edge \(e\in E(H)\). We distinguish three cases.

\medskip
\noindent
\textbf{Case 1:} \(e\cap A=\emptyset\).

Every vertex of \(e\) belongs to \(B\cup C\). If \(v\in B\), then
\(d(v)\leq r\), and therefore
$$
\omega(v)
=
\frac{r}{(r-1)d(v)}
\geq
\frac{1}{r-1}
\geq
\frac1r.
$$
If \(v\in C\), then \(d(v)\leq2r-1\), and therefore
$$
\omega(v)
=
\frac{2r-1}{r\,d(v)}
\geq
\frac1r.
$$
Thus, every vertex of \(e\) contributes at least \(1/r\). Since
\(|e|=r\), we obtain
$$
\omega(e)\geq r\cdot\frac1r=1.
$$

\medskip
\noindent
\textbf{Case 2:} \(e\cap A\neq\emptyset\) and
\(e\cap C\neq\emptyset\).

Choose
$
x\in e\cap A \text{ and }
y\in e\cap C
$.
Then
$$
d(x)\geq2r
\qquad\text{and}\qquad
d(y)\geq r+1.
$$

Suppose that some vertex
$
z\in e\setminus\{x,y\}
$ satisfies \(d(z)\geq2\). By \Cref{crownlem}, the edge \(e\) is then the
base of a copy of \(E_4^r\), contradicting the assumption that \(H\)
is \(E_4^r\)-free.

Therefore, every vertex in \(e\setminus\{x,y\}\) has degree one.
There are \(r-2\) such vertices, each belonging to \(B\), and each has
weight
$$
\frac{r}{(r-1)\cdot1}
=
\frac{r}{r-1}.
$$
Consequently,
$$
\omega(e)
\geq
(r-2)\frac{r}{r-1}
\geq1,
$$
where the final inequality follows from \(r\geq3\).

\medskip
\noindent
\textbf{Case 3:} \(e\cap A\neq\emptyset\) and \(e\cap C=\emptyset\).

In this case, every vertex of \(e\) belongs to \(A\cup B\).

Suppose first that \(e\) contains exactly one vertex \(x\in A\). Then
the remaining \(r-1\) vertices belong to \(B\). For every
\(v\in e\setminus\{x\}\),
$$
\omega(v)\geq\frac1{r-1}.
$$
Hence
$$
\omega(e)
\geq
(r-1)\frac1{r-1}
=
1.
$$

Suppose instead that \(e\) contains at least two vertices of \(A\).
Choose distinct vertices
$$
x,y\in e\cap A.
$$
If some vertex
$
z\in e\setminus\{x,y\}
$ had degree at least two, then
$$
d(x)\geq2r,\qquad
d(y)\geq2r\geq r+1,\qquad
d(z)\geq2,
$$
and \Cref{crownlem} would imply that \(e\) is the base of a copy of
\(E_4^r\), a contradiction.

Thus, every vertex in \(e\setminus\{x,y\}\) has degree one. Therefore,
$$
\omega(e)
\geq
(r-2)\frac{r}{r-1}
\geq1.
$$

The three cases prove the claimed bound. Summing over all
edges gives
$$
    |E(H)|
\leq
\sum_{e\in E(H)}\omega(e).
$$

We now interchange the order of summation. Every vertex \(v\in B\)
belongs to exactly \(d(v)\) hyperedges and therefore contributes
$$
d(v)\omega(v)
=
d(v)\frac{r}{(r-1)d(v)}
=
\frac{r}{r-1}.
$$
Similarly, every vertex \(v\in C\) contributes
$$
d(v)\omega(v)
=
d(v)\frac{2r-1}{r\,d(v)}
=
\frac{2r-1}{r},
$$
while every vertex in \(A\) contributes zero. Hence
$$
\begin{aligned}
\sum_{e\in E(H)}\omega(e)
&=
\sum_{v\in V(H)}d(v)\omega(v)\\
&=
\frac{r}{r-1}|B|
+
\frac{2r-1}{r}|C|.
\end{aligned}
$$
Therefore we obtain
$$
|E(H)|
\leq
\frac{r}{r-1}|B|
+
\frac{2r-1}{r}|C|.
$$
To obtain the equivalent form, let $
s=|A|.$
Since
$$
|B|+|C|=n-s,
$$
we have
$$
\begin{aligned}
|E(H)|
&\leq
\frac{r}{r-1}|B|
+
\frac{2r-1}{r}|C|\\
&=
\frac{r}{r-1}(|B|+|C|)
+
\left(
\frac{2r-1}{r}-\frac{r}{r-1}
\right)|C|\\
&=
\frac{r(n-s)}{r-1}
+
\frac{r^2-3r+1}{r(r-1)}|C|.
\end{aligned}
$$
This completes the proof.
\end{proof}

The conference upper bound is now an immediate consequence.

\begin{proof}[Proof of \Cref{E4thm}]
For \(r\geq3\),
$$
\frac{r}{r-1}
\leq
\frac{2r-1}{r},
$$
since
$$
\frac{2r-1}{r}-\frac{r}{r-1}
=
\frac{r^2-3r+1}{r(r-1)}
>0.
$$
Therefore, by \Cref{crowndeg},
$$
\begin{aligned}
|E(H)|
&\leq
\frac{r}{r-1}|B|
+
\frac{2r-1}{r}|C|\\
&\leq
\frac{2r-1}{r}(|B|+|C|)\\
&\leq
\frac{(2r-1)n}{r}.
\end{aligned}
$$
Thus,
$$
\operatorname{ex}^{\mathrm{lin}}_r(n,E_4^r)
\leq
\frac{(2r-1)n}{r}.
$$
\end{proof}

\subsection{A lower-bound construction, \Cref{E4lower}}
\label{subsec:crown-lower-bound}

\begin{proof}[Proof of \Cref{E4lower}]
Set
$
q=r-1
$
and write
$$
n-r=q^2m
$$
for some positive integer \(m\).

Let \(\mathcal{A}\) be an affine plane of order \(q\). Recall that
\(\mathcal{A}\) has \(q^2\) points and \(q(q+1)\) lines, every line
contains \(q\) points, and its lines can be partitioned into
\(q+1=r\) parallel classes
$$
\mathcal{P}_1,\mathcal{P}_2,\ldots,\mathcal{P}_r.
$$
Each parallel class consists of \(q\) pairwise disjoint lines, while
any two lines belonging to different parallel classes intersect in
exactly one point.

Take \(m\) vertex-disjoint copies
$$
\mathcal{A}_1,\mathcal{A}_2,\ldots,\mathcal{A}_m
$$
of \(\mathcal{A}\), and introduce \(r\) new vertices
$
s_1,s_2,\ldots,s_r.
$

For every \(j\in[r]\) and every line \(L\) in the \(j\)-th parallel
class of one of the copies \(\mathcal{A}_i\), introduce the
\(r\)-element hyperedge
$
L\cup\{s_j\}.
$
Let \(H\) denote the resulting \(r\)-uniform hypergraph.

We first verify that \(H\) is linear. Consider two distinct hyperedges
$$
L\cup\{s_a\}
\qquad\text{and}\qquad
M\cup\{s_b\}.
$$

Suppose first that \(L\) and \(M\) belong to different affine-plane
copies. Then the corresponding hyperedges are disjoint when
\(a\neq b\), and intersect only in \(s_a\) when \(a=b\).

Now suppose that \(L\) and \(M\) belong to the same copy. If \(a=b\),
then \(L\) and \(M\) lie in the same parallel class and are disjoint.
Thus, the corresponding hyperedges intersect only in \(s_a\). If
\(a\neq b\), then \(L\) and \(M\) belong to different parallel
classes and intersect in exactly one point, while \(s_a\neq s_b\).
Therefore, any two hyperedges of \(H\) intersect in at most one vertex.
Hence \(H\) is linear.

We next show that \(H\) is \(E_4^r\)-free. Suppose, to the contrary,
that \(H\) contains a copy of \(E_4^r\). Let
$$
e=L\cup\{s_j\}
$$
be its base edge, and let \(f_1,f_2,f_3\) be the three pairwise
disjoint hyperedges intersecting \(e\).

Since \(f_1,f_2,f_3\) are pairwise disjoint, at most one of them can
contain \(s_j\). Relabelling if necessary, assume that
$$
s_j\notin f_2\cup f_3.
$$
Both \(f_2\) and \(f_3\) intersect \(e\). Since neither contains
\(s_j\), each intersects the line \(L\) in an affine-plane point.
Consequently, their underlying lines belong to the same affine-plane
copy as \(L\).

Since \(f_2\) and \(f_3\) are disjoint, they cannot contain the same
new vertex \(s_a\). Therefore, their underlying lines belong to
different parallel classes. However, any two lines belonging to
different parallel classes of an affine plane intersect in exactly
one point. Hence
$$
f_2\cap f_3\neq\emptyset,
$$
contradicting the assumption that \(f_1,f_2,f_3\) are pairwise
disjoint. Therefore, \(H\) is \(E_4^r\)-free.

Finally,
$$
|V(H)|
=
q^2m+r
=
n.
$$
Each of the \(r\) parallel classes contributes \(qm\) hyperedges, and
therefore
$$
|E(H)|
=
rqm
=
r(r-1)m.
$$
Since
$$
n-r=(r-1)^2m,
$$
we obtain
$$
|E(H)|
=
r(r-1)m
=
\frac{r(n-r)}{r-1}.
$$
Thus,
$$
\operatorname{ex}^{\mathrm{lin}}_r(n,E_4^r)
\geq
\frac{r(n-r)}{r-1}.
$$
\end{proof}

\subsection{Consequences of the degree-sensitive bound}

We now prove the special case of \Cref{crownconj} stated in
\Cref{crowncor}.

\begin{proof}[Proof of \Cref{crowncor}]
Under the hypothesis of the corollary, no vertex has degree between
\(r+1\) and \(2r-1\). Therefore, \(C=\emptyset\). By \Cref{crowndeg},
$$
|E(H)|\leq \frac{r}{r-1}|B|.
$$
Since
$$
V(H)=A\mathbin{\dot\cup}B
\qquad\text{and}\qquad
s=|A|,
$$
we have \(|B|=n-s\). Consequently,
$$
|E(H)|\leq \frac{r(n-s)}{r-1}.
$$
\end{proof}

\begin{remark}\label{crownconstructionremark}
The construction in \Cref{E4lower} is consistent with \Cref{crownconj}
and attains equality when
$$
n-r=(r-1)^2m
$$
with \(m\geq3\). Indeed, every ordinary affine-plane point has degree
exactly \(r\), while each of the \(r\) added vertices has degree
\((r-1)m\). For \(m\geq3\), \((r-1)m\geq2r\), so the \(r\) added
vertices are precisely the vertices counted by \(s\). Hence \(s=r\),
and
$$
|E(H)|
=
\frac{r(n-r)}{r-1}
=
\frac{r(n-s)}{r-1}.
$$
More generally, a disjoint union of \(t\) such constructions has
\(s=tr\), which may be larger than \(r\), and still satisfies
$$
|E(H)|=\frac{r(n-s)}{r-1}.
$$
\end{remark}

\section{Proof of \Cref{P4thm}}

In this section, we prove the sharp upper bound for $P_4^r$ and characterize all
hypergraphs attaining equality.


For a hypergraph $H$, let $L(H)$ denote its \emph{line graph}. The
vertices of $L(H)$ are the hyperedges of $H$, and two vertices of $L(H)$
are adjacent if and only if the corresponding hyperedges intersect.

\begin{lemma}\label[lemma]{lem:path-line-graph}
Let $H$ be a linear $r$-uniform hypergraph. Four distinct hyperedges
$f_1,f_2,f_3,f_4$ form a copy of $P_4^r$, in this order, if and only if
the corresponding vertices induce a path on four vertices in $L(H)$.
\end{lemma}

\begin{proof}
Suppose first that $f_1,f_2,f_3,f_4$ form a copy of $P_4^r$. Then
$f_i\cap f_{i+1}\neq\emptyset$ for $i\in[3]$, while
$f_i\cap f_j=\emptyset$ whenever $|i-j|\geq 2$. Therefore, the
corresponding vertices induce a path on four vertices in $L(H)$.

Conversely, suppose that the vertices corresponding to
$f_1,f_2,f_3,f_4$ induce a path in this order. Thus, consecutive
hyperedges intersect and nonconsecutive hyperedges are disjoint. Since
$H$ is linear, every nonempty intersection between two of these
hyperedges consists of exactly one vertex. Moreover, the three
consecutive intersection vertices are distinct. Indeed, if two of them
were equal, then a pair of nonconsecutive hyperedges would intersect.
Hence, $f_1,f_2,f_3,f_4$ have exactly the incidence pattern of the
$r$-uniform expansion of a path with four edges, and therefore form a
copy of $P_4^r$.
\end{proof}

A graph containing no induced path on four vertices is called a
\emph{cograph}. We use the following standard decomposition of connected
cographs due to Seinsche~\cite{seinsche1974property}. We include a proof
for completeness.

\begin{lemma}
\label{lem:cograph-complement}
Let $G$ be a connected graph with at least two vertices and containing
no induced path on four vertices, that is $G$ is a connected cograph. Then $\overline{G}$ is disconnected.
\end{lemma}

\begin{proof}
We prove the statement by induction on $|V(G)|$. The result is immediate
when $|V(G)|=2$ (base case).

Now suppose $|V(G)|\geq 3$. Assume that the lemma is true for any connected graph with at least $2$ vertices, containing no induced path on four vertices and with less than $|V(G)|$ many vertices (induction hypothesis). Let $T$ be a spanning tree of $G$, and let
$x$ be a leaf of $T$. Clearly, $G-x$ is a connected cograph.
Therefore, by the induction hypothesis,
    $\overline{G-x}$ is disconnected.

Let $C_1,C_2,\ldots,C_t, t\geq 2$,
be the connected components of $\overline{G-x}$. Since there are no edges
of $\overline{G-x}$ between two distinct components, every vertex of
$C_i$ is adjacent in $G$ to every vertex of $C_j$ whenever $i\neq j$.

Suppose, for contradiction, that $\overline{G}$ is connected. Since the
only vertex added to $\overline{G-x}$ is $x$, the vertex $x$ must have a
neighbour in $\overline{G}$ in every component $C_i$. Equivalently, for
every $i\in[t]$, there exists a vertex of $C_i$ which is not adjacent to
$x$ in $G$.

On the other hand, $G$ is connected, so $x$ has a neighbor in $G$.
Without loss of generality, suppose that $x$ has a neighbor in $C_1$, say $y$. Therefore,
    $y\in C_1\cap N_G(x).$
As observed above, there also exists
    $w\in C_1\setminus N_G(x)$
Since $t\geq 2$, choose $z\in C_2\setminus N_G(x).$

The components $C_1$ and $C_2$ are complete to each other in $G$, that is in $G$, every vertex of $C_1$ is adjacent to every vertex in $C_2$.
Therefore, $yz,zw\in E(G).$
Since $y$ is adjacent to $x$ in $G$, $xy\in E(G),$ whereas $xz,xw\notin E(G).$

If $yw\notin E(G)$, then the four vertices $x,y,z,w$ result an induced path on four vertices, a contradiction. Hence
    $yw\in E(G).$ The same argument applies to every
$$
    y\in C_1\cap N_G(x)
    \quad\text{and}\quad
    w\in C_1\setminus N_G(x).
$$
Therefore, $C_1\cap N_G(x)$
is complete in $G$ to $C_1\setminus N_G(x)$.

Let $A=C_1\cap N_G(x).$ The set $A$ is nonempty. Moreover, every vertex of $A$ is adjacent in
$G$ to every vertex outside $A$: it is adjacent to $x$ by the definition
of $A$, to every vertex of $C_1\setminus A$ by the preceding argument,
and to every vertex of $C_2\cup\cdots\cup C_t$ because distinct
components of $\overline{G-x}$ are complete to each other in $G$.

Thus, there is no edge of $\overline{G}$ between $A$ and
$V(G)\setminus A$. This contradicts the assumption that
$\overline{G}$ is connected. Therefore, $\overline{G}$ is disconnected.
\end{proof}

The form of Lemma~\ref{lem:cograph-complement} that we will use is the
following immediate consequence.

\begin{corollary}
\label[corollary]{joincor}
Let $G$ be a connected graph with at least two vertices and containing
no induced path on four vertices, that is $G$ is a connected cograph. Then there exists a nontrivial partition
    $V(G)=A\mathbin{\dot\cup}B$ such that every vertex of $A$ is adjacent to every vertex of $B$.
Equivalently,
    $$G=G[A]\vee G[B],$$
where $\vee$ denotes the graph join.
\end{corollary}

\begin{proof}
By Lemma~\ref{lem:cograph-complement}, the complement $\overline{G}$ is
disconnected. Let $A$ be the vertex set of one connected component of
$\overline{G}$ and let $B=V(G)\setminus A.$
Since $\overline{G}$ is disconnected, both $A$ and $B$ are nonempty.

There is no edge of $\overline{G}$ between $A$ and $B$. Hence every
possible edge between $A$ and $B$ is present in $G$. Therefore every
vertex of $A$ is adjacent to every vertex of $B$, as required.
\end{proof}

Combining \Cref{lem:path-line-graph,joincor}, if $H$
is connected and $P_4^r$-free, then its hyperedges admit a nontrivial
partition into two classes such that every hyperedge in one class
intersects every hyperedge in the other. This gives the following
inequality.

\begin{lemma}\label[lemma]{joinineq}
Let $H$ be a connected linear $r$-uniform $P_4^r$-free hypergraph. Let
$
m=|E(H)|$ and $k=\Delta(H).$
Then $m\leq rk.$
Moreover, if $m=rk$, then every nonisolated vertex of $H$ has degree
$k$.
\end{lemma}

\begin{proof}
If $m=1$, then $k=1$ and the inequality follows from $r\geq 2$. Also,
equality is impossible in this case. Thus, assume that $m\geq 2$.

By \Cref{lem:path-line-graph}, $L(H)$ contains no induced path on four
vertices. Since $H$ is connected, $L(H)$ is connected. By
\Cref{joincor}, there exists a partition
$$
E(H)=\mathcal{A}\sqcup\mathcal{B},
\qquad
|\mathcal{A}|=s,
\qquad
|\mathcal{B}|=t,
$$
where $s,t>0$, such that every hyperedge of $\mathcal{A}$ intersects every
member of $\mathcal{B}$. For each $x\in V(H)$, define
$$
a_x=|\{A\in\mathcal{A}:x\in A\}|,
\qquad
b_x=|\{B\in\mathcal{B}:x\in B\}|,
\qquad
d_H(x)=a_x+b_x.
$$
Counting incidences with the two edge classes gives
$$
\sum_{x\in V(H)}a_x=rs
\qquad\text{and}\qquad
\sum_{x\in V(H)}b_x=rt.
$$
Every pair of hyperedges $(A,B)\in\mathcal{A}\times\mathcal{B}$ has exactly one
common vertex: it has at least one common vertex by the choice of the
partition, and at most one by linearity. Hence,
$$
\sum_{x\in V(H)}a_xb_x=st.
$$

We now compare the number of cross-pairs between $\mathcal{A}$ and
$\mathcal{B}$ with the maximum degree of $H$.
Since, $\sum_{x\in V(H)}a_xb_x=st$, $a_xb_x$ can be viewed as the number of pairs
$(A,B)\in\mathcal{A}\times\mathcal{B}$ whose unique intersection occurs
at the vertex $x$.

Recall that $\Delta(H) = k$. Our aim is to bound $a_xb_x$ locally using the degree
    $d_H(x)=a_x+b_x\leq k.$
For this purpose, put
    $p=\dfrac{s}{m}$ and $q=\dfrac{t}{m}.$
Thus $p+q=1$. 

The choice of these two quantities is motivated by the fact
that the total $\mathcal{A}$ and $\mathcal{B}$ incidences are
    $\sum_x a_x=rs$ and $\sum_x b_x=rt,$
and we want to obtain the term $rst/m$ after summing a suitable linear
combination of $a_x$ and $b_x$.

\begin{claim}\label[claim]{cl6.5}
    For a vertex $x$ with $d_H(x) >0$,  $\dfrac{a_xb_x}{d_H(x)}
    \leq
    q^2a_x+p^2b_x.$
\end{claim}
   \begin{proof} Recall that $p+q =1$, therefore $p^2+q^2 -1 = -2pq$. Thus
\begin{align*}
q^2a_x+p^2b_x-\frac{a_xb_x}{d_H(x)}
&=
q^2a_x+p^2b_x-\frac{a_xb_x}{a_x+b_x}
\\
&=
\frac{
q^2a_x(a_x+b_x)
+
p^2b_x(a_x+b_x)
-
a_xb_x
}{d_H(x)}\\
&=
\frac{
q^2a_x^2+p^2b_x^2+
(q^2+p^2-1)a_xb_x
}{d_H(x)}\\
&=
\frac{q^2a_x^2-2pqa_xb_x+p^2b_x^2}{d_H(x)}
=
\frac{(qa_x-pb_x)^2}{d_H(x)}
\geq 0.
\end{align*}\end{proof}
Now, since $d_H(x)\leq k$, we also have
    $\dfrac{a_xb_x}{k}
    \leq
    \dfrac{a_xb_x}{d_H(x)}.$
Thus from \Cref{cl6.5} we get
$$
    \frac{a_xb_x}{k}
    \leq
    \frac{a_xb_x}{d_H(x)}
    \leq
    q^2a_x+p^2b_x.
$$
Thus, summing over all vertices with $d_H(x)>0$ we get,
$$
\frac{1}{k}\sum_{x:d_H(x)>0}a_xb_x
\leq
q^2\sum_{x:d_H(x)>0}a_x
+
p^2\sum_{x:d_H(x)>0}b_x=
q^2rs+p^2rt.
$$
Using $\sum_xa_xb_x=st,$ we obtain
    $\dfrac{st}{k}
    \leq
    q^2rs+p^2rt.$
We now simplify the right-hand side. Since
    $p=\dfrac{s}{m},$ and $q=\dfrac{t}{m},$
and $s+t=m$, we have
$$
q^2rs+p^2rt
=
r\left(
\frac{t^2s}{m^2}
+
\frac{s^2t}{m^2}
\right)=
\frac{rst(s+t)}{m^2}=
\frac{rst}{m}.
$$
Therefore, $\dfrac{st}{k}\leq\dfrac{rst}{m}.$
Since $s,t>0$, we may cancel $st$, obtaining
$\dfrac{1}{k}\leq\dfrac{r}{m}.$ Therefore,
    $m\leq rk.$
    
\vspace{2mm} 

For the equality case, observe that, for every non-isolated vertex $x$,
$$
q^{2}a_x+p^{2}b_x-\frac{a_xb_x}{k}
=
\frac{(qa_x-pb_x)^{2}}{d_H(x)}
+a_xb_x\left(\frac{1}{d_H(x)}-\frac{1}{k}\right).
$$
Both terms on the right are nonnegative. If $m=rk$, equality forces both
terms to vanish. Thus $qa_x=pb_x$. Since $p,q>0$ and
$a_x+b_x=d_H(x)>0$, we have $a_xb_x>0$. Consequently,
$a_xb_x\left(\dfrac{1}{d_H(x)}-\dfrac{1}{k}\right)=0$ implies $d_H(x)=k$. Hence $H$ is $k$-regular.
\end{proof}

We next give the maximum-degree bound that will be combined with
\Cref{joinineq}.

\begin{lemma}\label[lemma]{maxtrace}
Let $H$ be a linear $r$-uniform $P_4^r$-free hypergraph with
$\delta(H)\geq 2$. If $k=\Delta(H)>r,$ then
$k\leq 2r-1.$
\end{lemma}

\begin{proof}
Choose $v\in V(H)$ with $d(v)=k$, and let $e_1,e_2,\ldots,e_k$
be the hyperedges containing $v$. By linearity, $e_i\cap e_j=\{v\}
\text{ for all distinct }i,j\in[k].$
We call a hyperedge \emph{external} if it does not contain $v$. For an
external hyperedge $f$, define its trace on the maximum star by
$$
T(f)=\{i\in[k]:f\cap e_i\neq\emptyset\}.
$$
Every index in $T(f)$ uses a different vertex of $f$. Indeed, if the
same vertex of $f$ belonged to both $e_i$ and $e_j$, then it would
belong to $e_i\cap e_j=\{v\}$ (due to linearity), contradicting that $f$ is external.
Therefore, $|T(f)|\leq r.$

Let $i\in[k]$ and
choose $x\in e_i\setminus\{v\}$. Since $\delta(H)\geq 2$, there exists
a hyperedge $f\neq e_i$ containing $x$. This hyperedge cannot contain
$v$, since otherwise $f$ and $e_i$ would intersect in both $x$ and
$v$, contradicting linearity. Thus, $f$ is external and $i\in T(f)$.
Therefore, $\bigcup_fT(f)=[k],$
where the union is over all external hyperedges. Therefore, the family of traces covers $[k]$.

Since $k>r$, the traces cannot be linearly ordered by inclusion. Indeed,
if they were, a largest trace would contain their union $[k]$, which is
impossible because every trace has cardinality at most $r$. Hence,
there exist external hyperedges $f$ and $g$ with incomparable traces.

\begin{claim}
    If $f$ and $g$ are external hyperedges with incomparable traces, then $f\cap g \neq \emptyset$ and $T(f)\cup T(g) = [k]$.
\end{claim}
\begin{proof} Set
$A=T(f)$ and $B=T(g)$. Choose $i\in A\setminus B$ and $j\in B\setminus A.$
If $f\cap g=\emptyset$, then the hyperedges
$f,e_i,e_j,g$
form a copy of $P_4^r$. This is a contradiction.
Therefore, $f\cap g\neq\emptyset.$

Suppose now that there exists
$l\in[k]\setminus(A\cup B)$. Then $
e_l,e_i,f,g$ forms a copy of $P_4^r$. This is again a contradiction. Hence, $
A\cup B=[k].$\end{proof}

Since $|A|,|B|\leq r$, we obtain $k\leq 2r$.

It remains to rule out $k=2r$. In this case, it is easy to check that we must have $|A|=|B|=r$ and
$A\cap B=\emptyset$.
Since $f$ has $r$ vertices and meets the $r$ star edges indexed by
$A$, every vertex of $f$ lies in exactly one of these star edges.
Similarly, every vertex of $g$ lies in exactly one of the star edges
indexed by $B$. 

Let $z\in f\cap g$. Then $z\in e_a\cap e_b$
for some $a\in A$ and $b\in B$. Since $A\cap B=\emptyset$, we have
$a\neq b$, and therefore $e_a\cap e_b=\{v\}$. This gives $z=v$, which
is impossible because $f$ and $g$ are external. Thus, $k\neq 2r$, and
hence $k\leq 2r-1$.
\end{proof}

\subsection{Proof of the inequality}

Suppose the inequality is false. Assume $H$ to be minimal counterexample. That is, $H$ is a $P_4^r$-free linear $r$-uniform hypergraph on $n$ vertices and $|E(H)| > \dfrac{(r+1)n}{r}$.

Similar to the proof of \Cref{b4thm}, we can get that $H$ must be connected with $\delta(H) \geq 2$ and $\Delta(H) \geq r+2$.
Let $k=\Delta(H)$, therefore
$k\geq r+2.$
By \Cref{maxtrace} we get, $r+2\leq k\leq 2r-1.$

When $r=2$, this interval is empty, and we immediately obtain a
contradiction. Thus, we may assume that $r\geq 3$.

The $k$ hyperedges through a vertex of degree $k$ intersect pairwise
exactly in their common centre. Therefore, their union contains
exactly $1+k(r-1)$ vertices, and hence
$$
n\geq 1+k(r-1).
$$
It follows that
$$
|E(H)|
>\frac{r+1}{r}n \geq
\frac{r+1}{r}\bigl(1+k(r-1)\bigr)= rk-\frac{k-r-1}{r}.
$$
Since $r+2\leq k\leq 2r-1$, we have
$$
0<\frac{k-r-1}{r}\leq\frac{r-2}{r}<1.
$$
Thus, $|E(H)|>rk-1$, and the integrality of $|E(H)|$ gives $
m\geq rk$, where $|E(H)| = m$.
On the other hand, \Cref{joinineq} gives $m\leq rk$.
Therefore, $m=rk.$
By the equality statement in \Cref{joinineq}, $H$ is
$k$-regular (since $\delta(H) \geq 2$, it does not contain any isolated vertices). Using the handshaking lemma we get,
$kn=r|E(H)|=r^2k,$
and hence $n=r^2.$

However, since $k\geq r+2$, the maximum star occupies at least
$$
1+(r+2)(r-1)=r^2+r-1>r^2
$$
vertices, a contradiction. Therefore, every $n$-vertex linear
$r$-uniform $P_4^r$-free hypergraph satisfies
$$
|E(H)|\leq\frac{(r+1)n}{r}.
$$

\subsection{Characterizing the extremal hypergraphs}

We now characterize all hypergraphs attaining equality. 

\vspace{2mm}

\noindent\textbf{$\bullet$ Only If Part}

\noindent It is enough to consider the connected case. Indeed, if
$H_1,\ldots,H_t$ are the nontrivial connected components of $H$, then
$$
|E(H)|=\sum_{i=1}^t |E(H_i)|
\leq \frac{r+1}{r}\sum_{i=1}^t |V(H_i)|
\leq \frac{r+1}{r}n.
$$
Hence equality for $H$ implies that $H$ has no isolated vertices and
that every connected component also attains equality.

Thus, for the only-if part, we may assume that $H$ is connected and $|E(H)|=\frac{r+1}{r}n.$

\begin{lemma}\label[lemma]{p4comp}
Let $H$ be a connected linear $r$-uniform $P_4^r$-free hypergraph such
that $|E(H)|=\frac{(r+1)n}{r}.$
Then $H\cong S(2,r,r^2).$
\end{lemma}

\begin{proof}

Since $H$ is connected it does not contain isolated vertices. Suppose it contains degree $1$ vertices. If a hyperedge contained $q\geq 1$
vertices of degree one, deleting that hyperedge and these vertices and
then applying the already proved upper bound would give
$$
|E(H)|-1\leq\frac{r+1}{r}(n-q),
$$
which results $|E(H)| < \dfrac{(r+1)n}{r}$, a contradiction.

The average degree of $H$ is $
\dfrac{r|E(H)|}{n}=r+1.$ Let $k = \Delta(H)$.
We first prove that $k\leq r+1$. Suppose, for a contradiction, that
$k\geq r+2$. If $r=2$, then \Cref{maxtrace} gives
$k\leq 3$, an immediate contradiction. Thus, assume $r\geq 3$. By
\Cref{maxtrace}, $k\leq 2r-1$. The maximum-star count
gives
$$
|E(H)|
=\frac{r+1}{r}n
\geq
\frac{r+1}{r}\bigl(1+k(r-1)\bigr)
=
rk-\frac{k-r-1}{r}
>
rk-1.
$$
Therefore, $|E(H)|\geq rk$. By \Cref{joinineq}, $|E(H)|\leq rk$, and
hence $|E(H)|=rk$. Equality in \Cref{joinineq} shows that $H$ is
$k$-regular. Thus, $
kn=r|E(H)|=r^2k,$ so $n=r^2$. But the maximum star contains at least $1+(r+2)(r-1)>r^2$ vertices, a contradiction. Hence, $k\leq r+1$.

Since the average degree is $r+1$ and no vertex has degree greater than
$r+1$, every vertex has degree exactly $r+1$. Applying
\Cref{joinineq}, we get $
|E(H)|\leq r(r+1).$
Since $|E(H)|=\dfrac{(r+1)n}{r}$, it follows that $n\leq r^2.$
On the other hand, the $r+1$ hyperedges through any fixed vertex have
pairwise intersection exactly at that vertex, and therefore their union
has size $1+(r+1)(r-1)=r^2.$
Thus, $n\geq r^2$. Therefore,
$$
n=r^2
\qquad\text{and}\qquad
m=r(r+1).
$$

Every hyperedge contains $\binom{r}{2}$ unordered pairs of vertices,
and linearity guarantees that no pair of vertices is contained in two
distinct hyperedges. Moreover,
$$
|E(H)|\binom{r}{2}
=
r(r+1)\binom{r}{2}
=
\binom{r^2}{2}.
$$
Therefore, every pair of vertices is contained in exactly one
hyperedge. Hence, $H\cong S(2,r,r^2)$.
\end{proof}
Thus every connected component of $H$ is isomorphic to
$S(2,r,r^2)$, and hence $H$ is a disjoint union of such systems.

\noindent\textbf{$\bullet$ If Part}

\noindent Now suppose that $H$ is a disjoint union of copies of
$S(2,r,r^2)$. We show that $H$ is $P_4^r$-free and attains the bound. 
We first verify that these Steiner systems are indeed $P_4^r$-free.

\begin{lemma}\label[lemma]{p4steiner}
Every Steiner system $S(2,r,r^2)$ is $P_4^r$-free and has
$r(r+1)$ hyperedges. Therefore, any hypergraph that is disjoint union of $S(2,r,r^2)$ is $P_4^r$-free and has $\dfrac{(r+1)n}{r}$ hyperedges, where $n$ is the number of vertices.
\end{lemma}

\begin{proof}
Let $C\cong S(2,r,r^2)$. By \Cref{steinercount},
$|E(C)|=r(r+1).$
Suppose, for a contradiction, that $C$ contains a copy of $P_4^r$ with
hyperedges $e_1,e_2,e_3,e_4$ in this order. Let $v=e_3\cap e_4.$

Then $v\notin e_1$, and both $e_3$ and $e_4$ are hyperedges containing
$v$ and disjoint from $e_1$. This contradicts \Cref{lem:steiner}, which
states that there is a unique hyperedge containing $v$ and disjoint
from $e_1$. Hence, $C$ is $P_4^r$-free.

 Clearly, disjoint unions of copies of $C$ remain $P_4^r$-free. The edge
count is additive.
If there are $t$ disjoint copies of $C$, then $n = tr^2$, and $|E(H)| = tr(r+1)$, Therefore,
$$|E(H)| = \frac{n}{r^2}r(r+1) = \frac{n(r+1)}{r}.$$
\end{proof}

\section{Counterexamples to a structural claim in
\cite{zhang2025linear}}\label{counterex}

Let $H$ be a connected linear $4$-uniform $P_4^4$-free hypergraph with
$\delta(H)\geq 2$, let $v$ be a vertex of maximum degree, and let $S$
be the star formed by the hyperedges containing $v$. In the proposed
proof of the $4$-uniform case in~\cite{zhang2025linear}, it was claimed
that $V(S)=V(H)$, with separate arguments for $\Delta(H)=6$ and
$\Delta(H)\geq 7$. The following two hypergraphs show that this
structural assertion is false. The proof of \Cref{P4thm} is independent
of this assertion.

First we provide a counterexample for the $\Delta(H) = 6$ case.
\begin{center}
\begin{minipage}{\textwidth}
\centering
\begin{tikzpicture}[
    x=1cm,
    y=1cm,
    scale=0.95,
    star/.style={
        draw=blue!75!black,
        line width=0.9pt,
        line cap=round,
        line join=round
    },
    rightfamily/.style={
        draw=red!75!black,
        line width=0.9pt,
        line cap=round,
        line join=round
    },
    leftfamily/.style={
        draw=teal!70!black,
        line width=0.9pt,
        line cap=round,
        line join=round
    },
    vertex/.style={
        circle,
        fill=black,
        inner sep=1.45pt
    },
    vertexlabel/.style={
        font=\scriptsize,
        fill=white,
        inner sep=0.75pt,
        text=black
    }
]


\coordinate (v) at (0,0);

\coordinate (v11) at (0.65,1.13);
\coordinate (v12) at (1.30,2.25);
\coordinate (v13) at (1.95,3.38);

\coordinate (v21) at (1.30,0);
\coordinate (v22) at (2.60,0);
\coordinate (v23) at (3.90,0);

\coordinate (v31) at (0.65,-1.13);
\coordinate (v32) at (1.30,-2.25);
\coordinate (v33) at (1.95,-3.38);

\coordinate (v41) at (-0.65,-1.13);
\coordinate (v42) at (-1.30,-2.25);
\coordinate (v43) at (-1.95,-3.38);

\coordinate (v51) at (-1.30,0);
\coordinate (v52) at (-2.60,0);
\coordinate (v53) at (-3.90,0);

\coordinate (v61) at (-0.65,1.13);
\coordinate (v62) at (-1.30,2.25);
\coordinate (v63) at (-1.95,3.38);

\coordinate (xh) at (0,4.70);


\draw[star] (v)--(v11)--(v12)--(v13);
\draw[star] (v)--(v21)--(v22)--(v23);
\draw[star] (v)--(v31)--(v32)--(v33);
\draw[star] (v)--(v41)--(v42)--(v43);
\draw[star] (v)--(v51)--(v52)--(v53);
\draw[star] (v)--(v61)--(v62)--(v63);


\draw[rightfamily]
    plot[smooth,tension=0.70]
    coordinates {(xh) (v11) (v21) (v31)};

\draw[rightfamily]
    plot[smooth,tension=0.70]
    coordinates {(xh) (v12) (v22) (v32)};

\draw[rightfamily]
    plot[smooth,tension=0.70]
    coordinates {(xh) (v13) (v23) (v33)};


\draw[leftfamily]
    plot[smooth,tension=0.70]
    coordinates {(xh) (v61) (v51) (v41)};

\draw[leftfamily]
    plot[smooth,tension=0.70]
    coordinates {(xh) (v62) (v52) (v42)};

\draw[leftfamily]
    plot[smooth,tension=0.70]
    coordinates {(xh) (v63) (v53) (v43)};


\foreach \p in {
    v,
    v11,v12,v13,
    v21,v22,v23,
    v31,v32,v33,
    v41,v42,v43,
    v51,v52,v53,
    v61,v62,v63,
    xh%
}{
    \node[vertex] at (\p) {};
}


\node[vertexlabel,above=3pt]       at (v)   {$v$};
\node[vertexlabel,above=3pt]       at (xh)  {$x$};

\node[vertexlabel,right=2pt]       at (v11) {$v_{11}$};
\node[vertexlabel,right=2pt]       at (v12) {$v_{12}$};
\node[vertexlabel,above right=2pt] at (v13) {$v_{13}$};

\node[vertexlabel,above right=2pt] at (v21) {$v_{21}$};
\node[vertexlabel,above right=2pt] at (v22) {$v_{22}$};
\node[vertexlabel,right=3pt]       at (v23) {$v_{23}$};

\node[vertexlabel,below left=2pt]  at (v31) {$v_{31}$};
\node[vertexlabel,below left=2pt]  at (v32) {$v_{32}$};
\node[vertexlabel,below=2pt]       at (v33) {$v_{33}$};

\node[vertexlabel,below right=2pt] at (v41) {$v_{41}$};
\node[vertexlabel,below right=2pt] at (v42) {$v_{42}$};
\node[vertexlabel,below=2pt]       at (v43) {$v_{43}$};

\node[vertexlabel,above right=2pt] at (v51) {$v_{51}$};
\node[vertexlabel,above right=2pt] at (v52) {$v_{52}$};
\node[vertexlabel,left=3pt]        at (v53) {$v_{53}$};

\node[vertexlabel,right=2pt]       at (v61) {$v_{61}$};
\node[vertexlabel,right=2pt]       at (v62) {$v_{62}$};
\node[vertexlabel,above left=2pt]  at (v63) {$v_{63}$};

\end{tikzpicture}


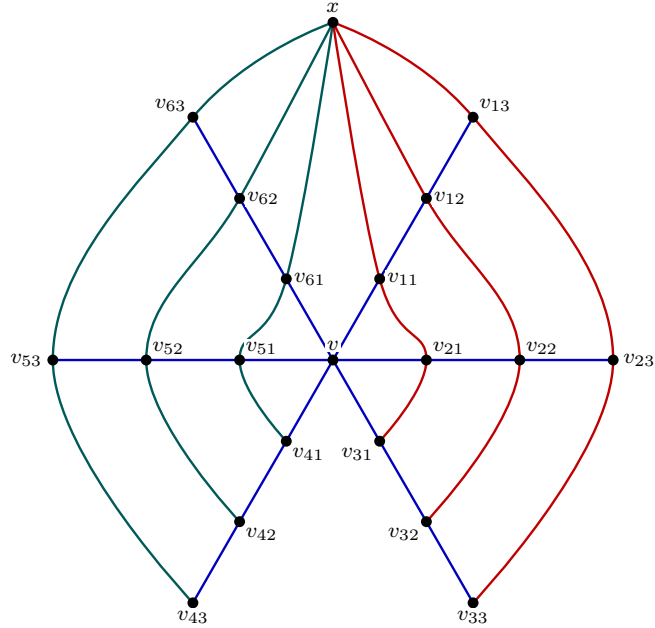
\captionof{figure}{A connected linear $4$-uniform $P_4^4$-free hypergraph with $\Delta(H)=6$ and $\delta(H)\geq 2$}
\label{ce2}
\end{minipage}
\end{center}

$H$ (in \Cref{ce2}) is a connected $4$-uniform linear $P_4^4$-free hypergraph with $\Delta(H) =6$. Let $v \in V(H)$ such that $d(v) = 6$ and $\delta(H) \geq 2$. Let $S$ be the star centered at $v$. We have the following three types of hyperedges;
$$
\begin{aligned}
E(H)={}&
\bigl\{\{v,v_{i1},v_{i2},v_{i3}\}: i\in[6]\bigr\}\\
&\cup
\bigl\{\{x,v_{1i},v_{2i},v_{3i}\}: i\in[3]\bigr\}\\
&\cup
\bigl\{\{x,v_{4i},v_{5i},v_{6i}\}: i\in[3]\bigr\}.
\end{aligned}
$$

Note that $V(H)\setminus V(S) = \{x\}$. Thus, $V(S) \neq V(H)$, contradicting the claim in \cite{zhang2025linear}. Now we provide a counterexample for the $\Delta(H) \geq 7$ case, specifically for $\Delta(H) = 7$.  
\begin{center}
\begin{minipage}{\textwidth}
\centering
\begin{tikzpicture}[
    x=1cm,
    y=1cm,
    scale=0.95,
    star/.style={
        draw=blue!75!black,
        line width=0.9pt,
        line cap=round,
        line join=round
    },
    rightfamily/.style={
        draw=red!75!black,
        line width=0.9pt,
        line cap=round,
        line join=round
    },
    leftfamily/.style={
        draw=teal!70!black,
        line width=0.9pt,
        line cap=round,
        line join=round
    },
    rowfamily/.style={
        draw=orange!85!black,
        line width=0.9pt,
        line cap=round,
        line join=round
    },
    columnfamily/.style={
        draw=violet!80!black,
        line width=0.9pt,
        line cap=round,
        line join=round
    },
    vertex/.style={
        circle,
        fill=black,
        inner sep=1.45pt
    },
    vertexlabel/.style={
        font=\scriptsize,
        fill=white,
        inner sep=0.75pt,
        text=black
    },
    edgelabel/.style={
        font=\small,
        fill=white,
        inner sep=0.7pt,
        text=blue!75!black
    }
]


\coordinate (v)   at (0,-0.35);

\coordinate (v11) at (0,0.85);
\coordinate (v12) at (0,2.05);
\coordinate (v13) at (0,3.25);

\coordinate (v21) at (1.12,0.32);
\coordinate (v22) at (2.25,0.95);
\coordinate (v23) at (3.42,1.60);

\coordinate (v31) at (1.08,-0.92);
\coordinate (v32) at (2.23,-1.48);
\coordinate (v33) at (3.40,-2.08);

\coordinate (v41) at (0.64,-1.48);
\coordinate (v42) at (1.20,-2.65);
\coordinate (v43) at (1.76,-3.84);

\coordinate (v51) at (-0.64,-1.48);
\coordinate (v52) at (-1.20,-2.65);
\coordinate (v53) at (-1.76,-3.84);

\coordinate (v61) at (-1.08,-0.92);
\coordinate (v62) at (-2.23,-1.48);
\coordinate (v63) at (-3.40,-2.08);

\coordinate (v71) at (-1.12,0.32);
\coordinate (v72) at (-2.25,0.95);
\coordinate (v73) at (-3.42,1.60);


\coordinate (x11) at (-4.95,6.35);
\coordinate (x12) at (-3.58,6.35);
\coordinate (x13) at (-2.21,6.35);

\coordinate (x21) at (-4.95,5.18);
\coordinate (x22) at (-3.58,5.18);
\coordinate (x23) at (-2.21,5.18);

\coordinate (x31) at (-4.95,4.01);
\coordinate (x32) at (-3.58,4.01);
\coordinate (x33) at (-2.21,4.01);


\draw[star] (v)--(v11)--(v12)--(v13);
\draw[star] (v)--(v21)--(v22)--(v23);
\draw[star] (v)--(v31)--(v32)--(v33);
\draw[star] (v)--(v41)--(v42)--(v43);
\draw[star] (v)--(v51)--(v52)--(v53);
\draw[star] (v)--(v61)--(v62)--(v63);
\draw[star] (v)--(v71)--(v72)--(v73);


\draw[rightfamily]
    plot[smooth,tension=0.70]
    coordinates {(v11) (v21) (v31) (v41)};

\draw[rightfamily]
    plot[smooth,tension=0.70]
    coordinates {(v11) (v22) (v32) (v42)};

\draw[rightfamily]
    plot[smooth,tension=0.70]
    coordinates {(v11) (v23) (v33) (v43)};


\draw[leftfamily]
    plot[smooth,tension=0.70]
    coordinates {(v11) (v71) (v61) (v51)};

\draw[leftfamily]
    plot[smooth,tension=0.70]
    coordinates {(v11) (v72) (v62) (v52)};

\draw[leftfamily]
    plot[smooth,tension=0.70]
    coordinates {(v11) (v73) (v63) (v53)};

%

\draw[rowfamily]
    (x11)--(x12)--(x13)
    .. controls (-1.00,6.35) and (0.95,4.50) ..
       (0.75,2.80)
    .. controls (0.62,2.35) and (0.30,2.08) ..
       (v12);

\draw[rowfamily]
    (x21)--(x22)--(x23)
    .. controls (-1.00,5.18) and (0.70,4.05) ..
       (0.62,2.72)
    .. controls (0.52,2.32) and (0.25,2.10) ..
       (v12);

\draw[rowfamily]
    (x31)--(x32)--(x33)
    .. controls (-1.05,4.01) and (0.48,3.40) ..
       (0.48,2.62)
    .. controls (0.40,2.28) and (0.20,2.12) ..
       (v12);

%

\draw[columnfamily]
    (x11)--(x21)--(x31)
    .. controls (-4.95,2.65) and (-2.65,2.35) ..
       (-1.10,2.55)
    .. controls (-0.60,2.62) and (-0.28,2.94) ..
       (v13);

\draw[columnfamily]
    (x12)--(x22)--(x32)
    .. controls (-3.58,2.85) and (-2.10,2.62) ..
       (-0.92,2.78)
    .. controls (-0.48,2.84) and (-0.22,3.05) ..
       (v13);

\draw[columnfamily]
    (x13)--(x23)--(x33)
    .. controls (-2.21,3.10) and (-1.55,2.98) ..
       (-0.72,3.02)
    .. controls (-0.38,3.04) and (-0.18,3.13) ..
       (v13);


\foreach \p in {
    v,
    v11,v12,v13,
    v21,v22,v23,
    v31,v32,v33,
    v41,v42,v43,
    v51,v52,v53,
    v61,v62,v63,
    v71,v72,v73,
    x11,x12,x13,
    x21,x22,x23,
    x31,x32,x33%
}{
    \node[vertex] at (\p) {};
}


\node[vertexlabel,below=3pt] at (v) {$v$};

\node[vertexlabel,above right=3pt] at (v11) {$v_{11}$};
\node[vertexlabel,right=5pt]       at (v12) {$v_{12}$};
\node[vertexlabel,right=5pt]       at (v13) {$v_{13}$};

\node[vertexlabel,above=2pt]       at (v21) {$v_{21}$};
\node[vertexlabel,above=2pt]       at (v22) {$v_{22}$};
\node[vertexlabel,above right=2pt] at (v23) {$v_{23}$};

\node[vertexlabel,above right=2pt] at (v31) {$v_{31}$};
\node[vertexlabel,above right=2pt] at (v32) {$v_{32}$};
\node[vertexlabel,below right=2pt] at (v33) {$v_{33}$};

\node[vertexlabel,left=3pt]        at (v41) {$v_{41}$};
\node[vertexlabel,left=3pt]        at (v42) {$v_{42}$};
\node[vertexlabel,below right=2pt] at (v43) {$v_{43}$};

\node[vertexlabel,above=2pt]      at (v71) {$v_{71}$};
\node[vertexlabel,above=2pt]      at (v72) {$v_{72}$};
\node[vertexlabel,above left=2pt] at (v73) {$v_{73}$};

\node[vertexlabel,above left=2pt] at (v61) {$v_{61}$};
\node[vertexlabel,above left=2pt] at (v62) {$v_{62}$};
\node[vertexlabel,below left=2pt] at (v63) {$v_{63}$};

\node[vertexlabel,right=3pt]      at (v51) {$v_{51}$};
\node[vertexlabel,right=3pt]      at (v52) {$v_{52}$};
\node[vertexlabel,below left=2pt] at (v53) {$v_{53}$};

\node[vertexlabel,above left=2pt]  at (x11) {$x_{11}$};
\node[vertexlabel,above=2pt]       at (x12) {$x_{12}$};
\node[vertexlabel,above right=2pt] at (x13) {$x_{13}$};

\node[vertexlabel,left=3pt]        at (x21) {$x_{21}$};
\node[vertexlabel,above right=2pt] at (x22) {$x_{22}$};
\node[vertexlabel,above right=3pt] at (x23) {$x_{23}$};

\node[vertexlabel,below left=2pt]  at (x31) {$x_{31}$};
\node[vertexlabel,below left=2pt]  at (x32) {$x_{32}$};
\node[vertexlabel,below right=2pt] at (x33) {$x_{33}$};

\end{tikzpicture}


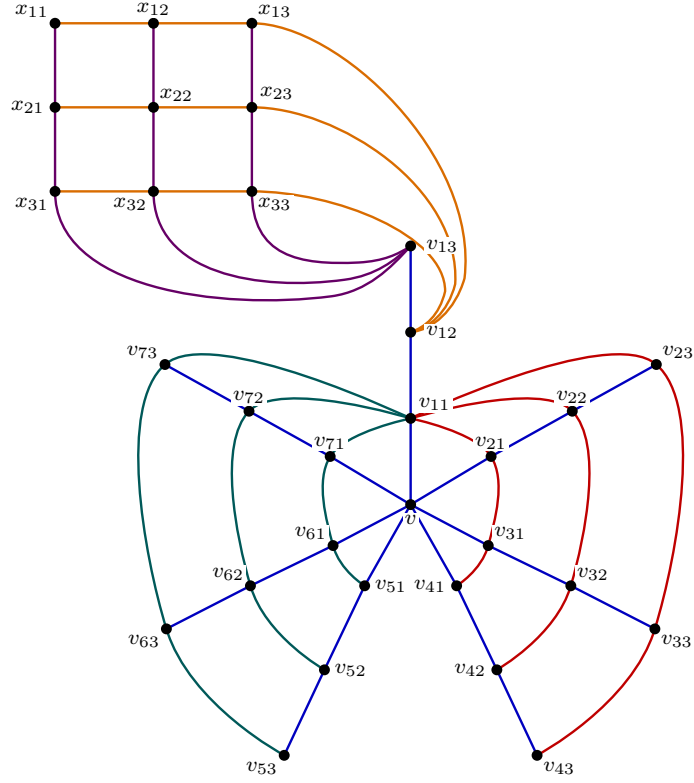
\captionof{figure}{
A connected linear $4$-uniform $P_4^4$-free hypergraph with
$\Delta(H)=7$, $\delta(H)\geq 2$}
\label{ce1}
\end{minipage}
\end{center}

  $H$ (in \Cref{ce1}) is a connected $4$-uniform linear $P_4^4$-free hypergraph with $\Delta(H) =7$. Let $v \in V(H)$ such that $d(v) = 7$ and $\delta(H) \geq 2$. Let $S$ be the star centered at $v$. We have the following five types of hyperedges;
  $$
\begin{aligned}
E(H)= {}&
\bigl\{\{v,v_{i1},v_{i2},v_{i3}\}: i\in[7]\bigr\}\\
&\cup
\bigl\{\{v_{11},v_{2i},v_{3i},v_{4i}\}: i\in[3]\bigr\}\\
&\cup
\bigl\{\{v_{11},v_{5i},v_{6i},v_{7i}\}: i\in[3]\bigr\}\\
&\cup
\bigl\{\{x_{i1},x_{i2},x_{i3},v_{12}\}: i\in[3]\bigr\}\\
&\cup
\bigl\{\{x_{1i},x_{2i},x_{3i},v_{13}\}: i\in[3]\bigr\}.
\end{aligned}
$$

Note that $V(H)\setminus V(S) = \{x_{ij} \mid i,j \in [3]\}$. Thus, $V(S) \neq V(H)$, contradicting the claim in \cite{zhang2025linear}.

\section*{Conclusion and Future Directions}

In this paper, we studied linear Tur\'an numbers of \(r\)-uniform linear
hypertrees. We gave a general lower-bound construction, determined the
extremal behaviour of \(B_4^r\), and obtained a degree-sensitive upper
bound and a lower bound for the crown \(E_4^r\). The crown bounds
motivate \Cref{crownconj}, and \Cref{crowndeg} proves this conjecture
whenever no vertex has degree between \(r+1\) and \(2r-1\). For the
linear path \(P_4^r\), we proved the sharp universal bound for every
\(r\geq 2\), and characterized equality by disjoint unions of Steiner
systems \(S(2,r,r^2)\).

A natural direction for future work is to prove \Cref{crownconj} in
full, close the constant-factor gap between the bounds in
\Cref{E4thm,E4lower}, and characterize the extremal \(E_4^r\)-free
hypergraphs. For \(B_4^r\) and \(P_4^r\), \Cref{b4thm,P4thm} respectively
determine the sharp universal coefficient and all equality cases. If
\(r^2\nmid n\), or if no Steiner system \(S(2,r,r^2)\) exists, equality
in the real-valued bound is impossible. Determining the exact integer
value of \(\operatorname{ex}^{\mathrm{lin}}_r(n,B_4^r)\) and \(\operatorname{ex}^{\mathrm{lin}}_r(n,P_4^r)\) for all such
remaining pairs \((n,r)\) is a finer problem.

More broadly, obtaining sharp bounds for longer linear paths \(P_k^r\)
is an interesting problem. The current general upper bound from Zhou
and Yuan~\cite{zhou2025turan} is likely far from optimal, especially
when compared with the sharp classical result for paths in graphs
\cite{gallai1959maximal}. The four-edge path is especially well-suited to the line-graph approach, since forbidding it is equivalent to requiring that the line graph be a cograph. It would be interesting to determine whether
related structural ideas can be used for longer paths.

\section*{Declaration of competing interest}

The authors declare that they have no known competing financial
interests or personal relationships that could have appeared to
influence the work reported in this paper.

\section*{Declaration of Generative AI and AI-Assisted Technologies}

During the preparation of this work, the authors used ChatGPT (OpenAI) as an auxiliary research tool to assist with computations related to the proof of \Cref{joinineq} and with the construction and verification of the counterexamples in \Cref{counterex}. All AI-assisted outputs were independently checked and revised by the authors, who take full responsibility for the correctness and content of the manuscript.

\bibliographystyle{unsrtnat}
\bibliography{references}

\end{document}